\documentclass[11pt]{amsart}

\usepackage{amsaddr}
\usepackage{caption}
\usepackage{enumerate}
\usepackage{float}
\usepackage{graphicx}
\usepackage{geometry} 
\usepackage{setspace}
\usepackage{subcaption}
\usepackage{tabu} 

\newcommand{\vx}{\vct{x}}
\newcommand{\vy}{\vct{y}}

\newcommand{\vu}{u}
\newcommand{\vvu}{\vvct{u}}

\newcommand{\valpha}{\boldsymbol\alpha}

\newcommand{\vkappa}{\boldsymbol\kappa}

\newcommand{\im}{\hat{\imath}}

\newcommand{\vct}[1]{\mathbf{#1}}
\newcommand{\vvct}[1]{\underline{#1}}
\newcommand{\ten}[1]{\mathbf{#1}}

\newtheorem{thm}{Theorem}[section]
\newtheorem{lem}[thm]{Lemma}

\title[New Higher-Order Mass-Lumped Tetrahedral Elements]{New Higher-Order Mass-Lumped Tetrahedral Elements for Wave Propagation Modelling*}
\author{S. Geevers$^1$, W.A. Mulder$^{2,3}$ \and J.J.W. van der Vegt$^1$}
\address{1. Department of Applied Mathematics, University of Twente, Enschede, the Netherlands (e-mail: s.geevers@utwente.nl, j.j.w.vandervegt@utwente.nl)}
\address{2. Shell Global Solutions International BV (e-mail: wim.mulder@shell.com)}
\address{3. Delft University of Technology}
\thanks{*This work was funded by the Shell Global Solutions International B.V. under contract no.~PT45999.}

\begin{document}
\maketitle

\begin{abstract}
We present a new accuracy condition for the construction of continuous mass-lumped elements. This condition is less restrictive than the one currently used and enabled us to construct new mass-lumped tetrahedral elements of degrees 2 to 4. The new degree-2 and degree-3 tetrahedral elements require 15 and 32 nodes per element, respectively, while currently, these elements require 23 and 50 nodes, respectively. The new degree-4 elements require 60, 61 or 65 nodes per element. Tetrahedral elements of this degree had not been found yet. We prove that our accuracy condition results in a mass-lumped finite element method that converges with optimal order in the $L^2$-norm and energy-norm. A dispersion analysis and several numerical tests confirm that our elements maintain the optimal order of accuracy and show that the new mass-lumped tetrahedral elements are more efficient than the current ones.
\end{abstract}

\section{Introduction}
Wave propagation modelling has many applications in the fields of structural mechanics, electromagnetism and geosciences. In many of these applications, waves need to be modelled on a large and complex 3D geometry that requires a fast and robust numerical algorithm. 

The oldest and most popular algorithm is the finite difference method, which approximates the wave field on a uniform grid. This method is relatively easy to implement and is very efficient on simple geometries. However, its accuracy quickly deteriorates if the grid points are not aligned with sharp material interfaces and boundaries of the domain. A good alignment is often not possible with uniform grids. 

Unstructured meshes, on the other hand, offer more geometric flexibility and can be properly aligned with many complex geometries. Such meshes can be used with finite element methods. While more difficult to implement and requiring more computations, the finite element method can remain accurate on very complex geometries when using a proper mesh. When applied with mass lumping, the finite element method can in such cases become more efficient than the finite difference method \cite{zhebel14}.

Mass lumping is important for applying the finite element method to wave propagation problems, since it allows for explicit time-stepping. When using an explicit time integration scheme, the finite element method requires the solution of a linear system $Mx=b$, with $M$ the mass matrix, at every time step. When using the classical finite element method, the mass matrix is large and sparse, but not (block)-diagonal. This makes the numerical scheme very inefficient for large-scale simulations. Mass lumping avoids this problem by lumping the mass matrix $M$ into a diagonal matrix. Usually, this is done with nodal basis functions and an inexact quadrature rule for $M$ of which the quadrature points coincide with the basis functions nodes.

For quadrilaterals and hexahedra, mass lumping is relatively straightforward and is accomplished by using tensor product basis functions and Gauss--Lobatto quadrature points. The resulting method is known as the spectral element method. Quadrilaterals and hexahedra, however, offer less geometric flexibility than triangles and tetrahedra. 

For linear triangular and tetrahedral elements, mass lumping is done using standard Lagrangian basis functions and a Newton--Cotes integration rule. For higher-degree triangular and tetrahedral elements, however, this approach results in instabilities, a singular mass matrix, or a suboptimal convergence rate. The Newton--Cotes rule for quadratic triangular elements, for example, has zero weights at the vertices, resulting in a singular mass matrix. This can be resolved by enriching the quadratic element space with a cubic bubble function that vanishes on all edges and by adding an additional node at the centre of the triangle \cite{fried75}. By enriching the element space with higher-degree bubble functions and combining it with a suitable quadrature rule, mass-lumped triangular elements were also obtained for degrees 3 \cite{cohen95, cohen01}, 4 \cite{mulder96}, 5 \cite{chin99}, 6 \cite{mulder13}, and 7 to 9 \cite{liu17, cui17}. For tetrahedra, mass lumping can be accomplished in a similar way by adding higher-degree face and internal bubble functions to the element space. So far, this has resulted in mass-lumped tetrahedral elements of degrees 2 \cite{mulder96} and 3 \cite{chin99}.

In this paper we show that the accuracy condition that was imposed on the quadrature rules of these higher-degree triangular and tetrahedral mass-lumped elements is too strong. This condition is that the quadrature rule of a degree-$p$ element should be exact for polynomials up to degree $p+p'-2$ \cite{ciarlet78}, where $p'>p$ is the highest polynomial degree of the functions in the enriched element space. Instead, we show that for $p\geq 2$ the quadrature rule only needs to be exact for functions in $\tilde U\otimes\mathcal{P}_{p-2}$, with $\tilde U$ the enriched element space and $\mathcal{P}_{p-2}$ the set of polynomials up to degree $p-2$. We prove that by satisfying this condition, the finite element method can maintain an optimal order of convergence in the $L^2$-norm and energy-norm.

This new accuracy condition enabled us to develop several new mass-lumped tetrahedral elements of degrees 2 to 4. The new elements of degree 2 and 3 require 15 and 32 nodes per element, respectively, while the current versions require 23 and 50 nodes, respectively. Our degree-4 elements require 60, 61 or 65 nodes. Mass-lumped tetrahedral elements of this degree had not been found yet. A dispersion analysis and various numerical tests confirm the optimal order of convergence of these methods and show that the new mass-lumped tetrahedral elements are significantly more efficient than the current ones.

Although this paper focuses on wave propagation problems, more generally, mass lumping is useful for solving any type of evolution problem that requires explicit time-stepping. It is also useful for efficiently computing higher-order derivatives, which appear, for example, in the Korteweg--de Vries equation \cite{minjeaud18}.

This paper is constructed as follows: In Section \ref{sec:FEM}, we present the scalar wave equation and the classical finite element method. In Section \ref{sec:ML}, we explain mass lumping. The stability is analyzed in Section \ref{sec:stability}. In Section \ref{sec:accuracy}, we present our new accuracy condition for the quadrature rule for the mass matrix and prove that, if this condition is satisfied, the mass-lumped finite element method can maintain an optimal order of convergence. This condition enabled us to derive several new mass-lumped tetrahedral elements of degrees 2 to 4, presented in Section \ref{sec:MLelements}. We analyze the dispersion properties of these new methods in Section \ref{sec:dispersion} and test the methods numerically in Section \ref{sec:tests}. In both sections we compare the new methods with existing finite element methods. Finally, we present our main conclusions in Section \ref{sec:conclusion}.

\section{The Scalar Wave Equation and Classical Finite Element Method}
\label{sec:FEM}
In this paper, we mainly focus on the scalar wave equation, which serves as a model problem for more complex wave problems such as the elastic wave equations and Maxwell's equations. Let $\Omega\subset\mathbb{R}^3$ be a three-dimensional open bounded domain, with Lipschitz boundary $\partial \Omega$, and let $(0,T)$ be the time domain. The scalar wave equation can be written as
\begin{subequations}
\begin{align}
\rho\partial_t^2 u &= \nabla\cdot c\nabla u + f &&\text{in }\Omega\times (0,T), \\
u &=0 &&\text{on }\partial\Omega, \\
u|_{t=0} &= u_0 && \text{in }\Omega, \\
\partial_tu|_{t=0} &= v_0 && \text{in }\Omega,
\end{align}
\label{eq:hypP}%
\end{subequations}
where $u:\Omega\times(0,T)\rightarrow\mathbb{R}$ is the unknown scalar field, $\nabla$ is the gradient operator, $\rho,c:\Omega\rightarrow\mathbb{R}^+$ are positive scalar fields, and $f:\Omega\times(0,T)\rightarrow\mathbb{R}$ is the source term. We assume that the parameters $\rho$ and $c$ are bounded by $\rho_0\leq \rho \leq\rho_1$ and $c_0\leq c\leq c_1$ for some positive scalars $\rho_0,\rho_1,c_0,c_1\in\mathbb{R}^+$.

This equation can be solved with the finite element method, which is based on the weak formulation of (\ref{eq:hypP}). Assume the initial conditions satisfy $u_0\in H^1_0(\Omega)$ and $v_0\in L^2(\Omega)$ and assume the source term satisfies $f\in L^2\big(0,T;L^2(\Omega)\big)$. Here, $L^2(\Omega)$ denotes the space of square integrable functions on $\Omega$, $H^1_0$ denotes the Sobolev space of functions on $\Omega$ that are zero on $\partial \Omega$ and have square-integrable weak derivatives, and $L^2(0,T;U)$, with $U$ a Banach space, denotes the Bochner space consisting of functions $f:(0,T)\rightarrow U$ such that $\|f(t)\|_U$ is square integrable in $(0,T)$. The weak formulation of (\ref{eq:hypP}) is finding $u\in L^2\big(0,T;H^1_0(\Omega)\big)$, with $\partial_tu\in L^2\big(0,T;L^2(\Omega)\big)$ and $\partial_t(\rho\partial_tu)\in L^2\big(0,T;H^{-1}(\Omega)\big)$, such that $u|_{t=0}=u_0$, $\partial_tu|_{t=0}=v_0$, and
\begin{align}
\langle\partial_t(\rho\partial_t u), w\rangle + a(u,w) &= (f,w)  &&\text{for all }w\in H^1_0(\Omega), \text{ a.e. }t\in(0,T).
\label{eq:hypPWF}%
\end{align}
Here, $\langle \cdot,\cdot \rangle$ denotes the pairing between $H^{-1}(\Omega)$ and $H_0^1(\Omega)$, $(\cdot,\cdot)$ denotes the $L^2(\Omega)$ inner product, and $a(\cdot,\cdot):H^1_0(\Omega)\times H^1_0(\Omega)\rightarrow\mathbb{R}$ is the elliptic operator given by
\begin{align*}
a(u,w):=\int_{\Omega} c\nabla u\cdot\nabla w \;dx.
\end{align*}
Because of the boundedness of $\rho$, it follows that the norm $\|u\|_{\rho}^2:=(\rho u,u)$ is equivalent to the standard $L^2(\Omega)$-norm. It can then be proven, in a way analogous to \cite[Chapter 3, Theorem 8.1]{lions72}, that (\ref{eq:hypPWF}) is well-posed and has a unique solution.


The solution of (\ref{eq:hypPWF}) can be approximated by the finite element method. Let $\mathcal{T}_h$ be a tetrahedral tesselation of $\Omega$, with $h$ the diameter of the smallest sphere that can contain each element in $\mathcal{T}_h$, and let $U_h$  denote the finite element space consisting of continuous functions that are polynomial of degree at most $p$ when restricted to a single element:
\begin{align*}
U_{h} =\{u\in H^1_0(\Omega) \;|\; u|_{e}\in\mathcal{P}_p(e) \text{ for all }e\in\mathcal{T}_h\},
\end{align*}
where $\mathcal{P}_p$ denotes the set of all polynomials of degree $p$ or less. The classical conforming finite element method is finding $u_h:[0,T]\rightarrow U_{h}$, such that $u_h|_{t=0}=\Pi_{h}u_0$, $\partial_tu_h|_{t=0}=\Pi_{h}v_0$, and 
\begin{align}
(\rho\partial_t^2 u_h,w) + a(u_h,w) &= (f,w) && \text{for all }w\in U_{h}, \text{ a.e. }t\in(0,T),
\label{eq:CFEM}
\end{align}
where $\Pi_{h}:L^2(\Omega)\rightarrow U_{h}$ is the weighted $L^2$ projection operator defined such that $(\rho\Pi_{h}u,w)=(\rho u,w)$ for all $w\in U_{h}$. 

This can be rewritten as a set of ODE's using a linear basis $\{w_i\}_{i=1}^n$ of $U_{h}$. For any function $u\in U_{h}$ we define $\vvct{u}\in\mathbb{R}^n$ as the vector of coefficients such that $u = \sum_{i=1}^n \underline{u}_i w_i$. The finite element method can then be formulated as solving $\vvu_h:[0,T]\rightarrow\mathbb{R}^n$, such that $\vvu_h|_{t=0}=\vvct{\Pi_{h}u_0}$, $\partial_t\vvu_h|_{t=0}=\vvct{\Pi_{h}v_0}$, and
\begin{align}
M\partial_t^2\vvu_h + A\vvu_h = \vvct{f}^* && \text{for a.e. }t\in(0,T),
\label{eq:ODE}
\end{align}
where $M,A\in\mathbb{R}^{n\times n}$ are the mass matrix and stiffness matrix, respectively, given by $M_{ij} := (\rho w_i,w_j)$, $A_{ij} := a(w_i,w_j)$ for all $i,j=1,\dots,n$, and $\vvct{f}^*\in L^2(0,T;\mathbb{R}^n)$ is the source vector, given by $\vvct{f}_i^*:=(f,w_i)$, for $i=1,\dots,n$, a.e. $t\in(0,T)$.

When using an explicit time integration scheme, a system of the form $M\vvct{x}=\vvct{b}$ needs to be solved at every time step. Typically, the mass matrix $M$ is large and sparse, but not (block)-diagonal, resulting in a very inefficient numerical scheme. A diagonal mass matrix can be obtained by a technique known as mass lumping. We will discuss this in the next section.

\section{Mass lumping}
\label{sec:ML}
Mass lumping is usually done with nodal basis functions and an inexact quadrature rule for the mass matrix. A diagonal matrix is obtained when the integration points coincide with the nodes of the basis functions. However, when using elements of degree $p\geq 2$, this technique does not result in a stable and accurate finite element scheme. For example, for standard quadratic Lagrangian basis functions combined with a Newton--Cotes quadrature rule, the weights at the vertices of the quadratic tetrahedral element become negative, resulting in unstable modes. 

To overcome such problems, the elements are enriched with higher-degree face and interior bubble functions. These enriched elements are still affine-equivalent to a reference element $\tilde e$. We can therefore write the discrete space in the form 
\begin{align*}
U_h=H_0^1(\Omega) \cap U(\mathcal{T}_h, \tilde U),
\end{align*}
where
\begin{align*}
U(\mathcal{T}_h,\tilde U) &:= \{u\in H^1(\Omega) \;|\; u\circ\phi_e \in \tilde U, \text{for all }e\in\mathcal{T}_h \},
\end{align*}
with $\phi_e:\tilde e\rightarrow e$ the reference-to-physical element mapping, and $\tilde U$ the reference space.  If $\tilde U=\mathcal{P}_p(\tilde e)$ we obtain the standard elements of degree $p$. To obtain enriched elements, we set $\tilde U=\mathcal{P}_p(\tilde e)\oplus \tilde{U}^+:=\{u \;|\; u=w+u^+ \text{ for some }w\in\mathcal{P}_p(\tilde e), u^+\in \tilde{U}^+\} $, with $\tilde{U}^+$ a space of higher-degree face and interior bubble functions.



A nodal basis and quadrature rule for $U_h$ can be constructed from a nodal basis and quadrature rule for the reference space $\tilde U$. In the next two subsections we will discuss this in more detail.

\subsection{Nodes and Nodal Basis Functions}
\label{sec:nodalBF}
A nodal basis for a space $U_h$ consists of a set of nodes $\mathcal{Q}_h$ and corresponding basis functions $\{w_{\vx}\}_{\vx\in\mathcal{Q}_h}$, such that $\mathrm{span}\{w_{\vx}\}_{\vx\in\mathcal{Q}_h} = U_h$ and $w_{\vx}(\vy) = \delta_{\vx\vy}$, for all $\vx,\vy\in\mathcal{Q}_h$, where $\delta_{\vx\vy}$ denotes the Kronecker delta. This means that each basis function equals one at one particular node and zero at all the other nodes.

A common way to construct such a nodal basis for the space $U(\mathcal{T}_h,\tilde U)$ is using a nodal basis $\{\tilde w_{\tilde{\vx}}\}_{\tilde{\vx}\in\tilde{\mathcal{Q}}}$ for the reference space $\tilde U$. The element nodes $\mathcal{Q}_e$ are obtained by mapping the reference nodes to the physical element: $\mathcal{Q}_e:=\{\phi_e(\tilde\vx)\}_{\tilde\vx\in\tilde{\mathcal{Q}}}$. The nodal basis functions of this element, $\{w_{e,\vx}\}_{\vx\in\mathcal{Q}_e}$, are obtained by mapping the reference basis functions to the physical element. We can write these functions as $w_{e,\vx}:=\tilde w_{\phi_e^{-1}(\vx)}\circ\phi_e^{-1}$. The set of global nodes $\mathcal{Q}_h$ is the union of all element nodes and the corresponding global basis functions are obtained by concatenating the corresponding element basis functions. Formally, we define the global nodal basis functions $\{w_{\vx}\}_{\vx\in\mathcal{Q}_h}$ as follows:
\begin{align}
w_{\vx}|_{e} &:= \begin{cases}
\tilde w_{\phi_e^{-1}(\vx)}\circ\phi_e^{-1}, &  e\in \mathcal{T}_\vx, \\
0, &\text{otherwise},
\end{cases}
\label{eq:nodalBFD}
\end{align}
for all $\vx\in\mathcal{Q}_h$, where $\mathcal{T}_{\vx}$ denotes the set of elements containing or adjacent to $\vx$. To ensure that these global basis functions are well-defined and continuous, we need to impose the following additional conditions on $\tilde{\mathcal{Q}}$ and $\{\tilde w_{\tilde{\vx}}\}_{\tilde{\vx}\in\tilde{\mathcal{Q}}}$:
\begin{align}
\tilde w_{\tilde\vx}|_{\tilde f} &= 0 &&\text{for all }\tilde f\in\tilde{\mathcal{F}}, \tilde\vx\in\tilde{\mathcal{Q}} \setminus \tilde f, 
\label{eq:confC1}
\end{align}
and
\begin{subequations}
\begin{align}
\tilde{\mathcal{Q}} &= s(\tilde{\mathcal{Q}}) &&\text{for all }s\in\mathcal{S},  \label{eq:confC2a}\\
\tilde w_{\tilde\vx} &= \tilde w_{s(\tilde\vx)} \circ s &&\text{for all }s\in\mathcal{S}, \label{eq:confC2b}
\end{align}
\label{eq:confC2}%
\end{subequations}
where $\tilde{\mathcal{F}}$ is the set of reference faces and $\mathcal{S}$ is the set of all affine mappings that map $\tilde e$ onto itself. Condition (\ref{eq:confC1}) implies that if a basis function is zero at the nodes on a face, then it should be zero on the entire face, and condition (\ref{eq:confC2}) implies that the set of element nodes and basis functions are symmetric and do not depend on the choice of $\phi_e$. A proof that $\{w_{\vx}\}_{\vx\in\mathcal{Q}_h}$ is indeed a set of well-defined and continuous nodal basis functions is given in Lemma \ref{lem:nodalBF} and Theorem \ref{thm:nodalBF}. 


It remains to incorporate the Dirichlet boundary condition $u_h|_{\partial\Omega}=0$. If $u_h\in U(\mathcal{T}_h,\tilde U)=\mathrm{span}\{w_{\vx}\}_{\vx\in\mathcal{Q}_h}$, then, because of (\ref{eq:confC1}), this condition is satisfied when $u_h=0$ at all nodes on $\partial\Omega$. A nodal basis for $U_h$ therefore consists of all interior nodes $\mathcal{Q}_h\setminus\partial\Omega$ and corresponding basis functions $\{w_{\vx}\}_{\vx\in\mathcal{Q}_h\setminus\partial\Omega}$.

\subsection{Quadrature Rule}
\label{sec:quadRule}
To obtain a diagonal mass matrix, we approximate the integrals with an inexact quadrature rule of which the integration points coincide with the nodes of the nodal basis.

Let $\mathcal{Q}_e:=\{\phi_e(\tilde\vx)\}_{\tilde\vx\in\tilde{\mathcal{Q}}}$ be the set of nodes on $e$, and let $\{\omega_{e,\vx}\}_{\vx\in\mathcal{Q}_e}$ be a set of corresponding weights. Together, the weights and nodes form a quadrature rule for the element. The quadrature rule is used to approximate the integrals of the mass matrix at the element as follows:
\begin{align}
(\rho u,w)_e = \int_e \rho uw\;dx \approx \sum_{\vx\in\mathcal{Q}_e} \omega_{e,\vx} \rho_e(\vx)u(\vx)w(\vx)=:(\rho u,w)_{\mathcal{Q}_e},
\label{eq:quadD}
\end{align}
where $\rho_e:=\rho|_e$ denotes the scalar field $\rho$ restricted to element $e$. We assume that $\rho$ is continuous within each element, which implies that the approximation above is well defined. The global product $(\rho u,w)$ is then approximated by
\begin{align}
\label{eq:quadP}
(\rho u,w) \approx (\rho u,w)_{\mathcal{Q}_h} :=  \sum_{e\in\mathcal{T}_h} (\rho u,w)_{\mathcal{Q}_e}.
\end{align}

Now let $w_{\vx},w_{\vy}$, with $\vx,\vy\in\mathcal{Q}_h$, be nodal basis functions as described in the previous subsection. The corresponding mass matrix entry is given by
\begin{align}
(\rho w_{\vx},w_{\vy})_{\mathcal{Q}_h} =\delta_{\vx\vy} \sum_{ e\in\mathcal{T}_\vx} \omega_{e,\vx}\rho_e({\vx}),
\label{eq:diagM2}
\end{align}
This implies that the mass matrix is diagonal with entries of the form $\sum_{ e\in\mathcal{T}_\vx } \omega_{e,\vx}\rho_e({\vx})$.

The quadrature rules can be constructed from a reference quadrature rule. This rule consists of the reference nodes $\tilde{\mathcal Q}$ and a set of weights $\{\tilde\omega_{\tilde\vx}\}_{\tilde\vx\in\tilde{\mathcal Q}}$ and approximates integrals on the reference element as follows:
\begin{align*}
\int_{\tilde e} \tilde\rho \tilde u\tilde w\;d\tilde x \approx \sum_{\tilde\vx\in\tilde{\mathcal Q}} \omega_{\tilde\vx}\tilde\rho(\tilde\vx)\tilde u(\tilde\vx)\tilde v(\tilde\vx) =: (\tilde\rho\tilde u,\tilde w)_{\tilde{\mathcal{Q}}}.
\end{align*}
We can use this to approximate the integral of the physical element by
\begin{align*}
(\rho u,w)_e = \frac{|e|}{|\tilde e|} \int_{\tilde e} \tilde\rho \tilde u\tilde w\;d\tilde x \approx   \frac{|e|}{|\tilde e|}  (\tilde\rho\tilde u,\tilde w)_{\tilde{\mathcal{Q}}},
\end{align*}
with $|e|$ the volume of $e$, $|\tilde e|$ the volume of $\tilde e$, and $\tilde \rho:=\rho\circ\phi_e$, $\tilde u:=u\circ\phi_e$, $\tilde w:=w\circ\phi_e$. This approximation is the same as (\ref{eq:quadD}) when $\omega_{e,\vx}=({|e|}/{|\tilde e|})\tilde\omega_{\phi_{e}^{-1}(\vx)}$.

Now that we have introduced the quadrature rules for the mass matrix, we can present the mass-lumped finite element method.

\subsection{Mass-Lumped Finite Element Method}
Assume $\rho\in\mathcal{C}^0(\mathcal{T}_h)$, $u_0\in H^1_0(\Omega)\cap\mathcal{C}^0_0(\Omega)$, $v_0\in \mathcal{C}^0_0(\Omega)$, and $f\in L^2\big(0,T;\mathcal{C}^0(\overline\Omega)\big)$. Here, $\mathcal{C}^0(\mathcal{T}_h)$ denotes the set of functions that are in $\mathcal{C}^0(\overline e)$ when restricted to $e$. The mass-lumped finite element method is finding $u_h:[0,T]\rightarrow U_h$, such that $u_h|_{t=0}=I_hu_0$, $\partial_tu_h|_{t=0}=I_hv_0$, and 
\begin{align}
(\rho\partial_t^2 u_h,w)_{\mathcal{Q}_h} + a(u_h,w) &= (f,w)_{\mathcal{Q}_h} && \text{for all }w\in U_h, \text{ a.e. }t\in(0,T),
\label{eq:MLFEM}
\end{align}
where $I_h:\mathcal{C}^0(\overline\Omega)\rightarrow U(\mathcal{T}_h,\tilde U)$ denotes the interpolation of a continuous function by a function in $U(\mathcal{T}_h,\tilde U)$ through the nodes of $\mathcal{Q}_h$.

To write this as a set of ODE's, let $\{\vx^{(i)}\}_{i=1}^n=\mathcal{Q}_h\setminus\partial\Omega$ be a numbering of all interior nodes, and define $w_i:=w_{\vx^{(i)}}$ for all $i=1,2,\dots,n$. Then the mass-lumped finite element method can be formulated as solving $\vvu_h:[0,T]\rightarrow\mathbb{R}^n$ such that $\vvu_h|_{t=0}=\vvct{I_hu_0}$, $\partial_t\vvu_h|_{t=0}=\vvct{I_hv_0}$, and 
\begin{align}
M\partial_t^2\vvu_h + A\vvu_h = \vvct{f}^* && \text{for a.e. }t\in(0,T),
\label{eq:ODEML}
\end{align}
where $M_{ij} := (\rho w_i,w_j)_{\mathcal{Q}_h}$, $A_{ij} := a(w_i,w_j)$ for all $i,j=1,\dots,n$, and $\vvct{f}_i^*:=(f,w_i)_{\mathcal{Q}_h}$, for $i=1,\dots,n$, a.e. $t\in(0,T)$. From (\ref{eq:diagM2}) it follows that $M$ is now a diagonal matrix that can be written as
\begin{align}
M_{ij} &= \delta_{ij}\sum_{e\in\mathcal{T}_{\vx^{(i)}}} \omega_{e,\vx^{(i)}}\rho_e({\vx^{(i)}}), &&i,j=1,\dots,n.
\label{eq:diagM}
\end{align}
This set of ODE's can be efficiently solved using an explicit time integration scheme such as the second-order leap-frog scheme or a higher-order Dablain scheme \cite{dablain86}, which is a type of Lax--Wendroff scheme \cite{lax64} for second-order wave equations.

In the next sections we analyze the stability and accuracy of the mass-lumped finite element method and derive conditions for the quadrature rules.

\subsection{Stability of the Mass-Lumped Finite Element Method}
\label{sec:stability}
To analyze the stability of the mass-lumped finite element method, we look at the behavior of the discrete energy. Consider the mass-lumped method given in (\ref{eq:MLFEM}) and substitute $w=\partial_t u$ to obtain
\begin{align*}
\partial_t E_h &=(f,\partial_t u)_{\mathcal{Q}_h} &&\text{for a.e. }t\in(0,T),
\end{align*}
where $E_h:=\frac12(\rho \partial_tu,\partial_tu)_{\mathcal{Q}_h}+\frac12a(u,u)$ is the discrete energy. This implies that the discrete energy remains bounded when the source term $f$ is bounded and that the discrete energy is conserved when there is no source term. 

For stability it then remains to show that the discrete energy is a well-defined energy. This means that $(\rho v,v)_{\mathcal{Q}_h}+ a(u,u)>0$ for all $u,v\in U_h$, $(u,v)\neq 0$, which is the case when $(\rho u,u)_{\mathcal{Q}_h}>0$ for any $u\in U_h$, $u\neq 0$. Since we can write $(\rho u,u)_{\mathcal{Q}_h}=\vvu^tM\vvu$, this is satisfied when $M$ is positive definite. From (\ref{eq:diagM}) it follows that this is the case when all weights of the quadrature rules are strictly positive, which is the case when the weights of the reference quadrature rule are strictly positive.

\section{Accuracy of the Mass-Lumped Finite Element Method}
\label{sec:accuracy}
\subsection{A Less Restrictive Condition on the Accuracy of the Quadrature Rule}
\label{sec:mainResult1}
Let $U_h=U(\mathcal{T}_h,\tilde U)$, with $\tilde U=\mathcal{P}_p(\tilde e)\oplus\tilde{U}^+$, be the finite element space constructed as in Section \ref{sec:ML}, where $p\geq 2$ denotes the degree of the finite element method and $\tilde{U}^+\subset\mathcal{P}_{p'}(\tilde e)$ is the space of higher-degree face and interior bubble functions.  Also, let the quadrature rule for the mass matrix be based on a reference element quadrature rule as described in Section \ref{sec:quadRule}. We will prove that an optimal convergence rate of the mass-lumped finite element method is obtained when all weights of the reference quadrature rule, $\{\tilde\omega_{\tilde\vx}\}_{\tilde\vx\in\tilde{\mathcal Q}}$, are strictly positive and
\begin{align}
\label{eq:quadC1}
\int_{\tilde e} \tilde f\;d\tilde x &=  \sum_{\tilde{\vx}\in\tilde{\mathcal{Q}}} \tilde{\omega}_{\tilde{\vx}}\tilde f(\vx) &&\text{for all }\tilde f\in\mathcal{P}_{p-2}(\tilde e)\otimes\tilde U,
\end{align}
where $\mathcal{P}_{p-2}(\tilde e)\otimes\tilde U:=\{f \;|\; f=wu \text{ for some }w\in\mathcal{P}_{p-2}(\tilde e), u\in\tilde U\}$. This means that the quadrature rule of the reference element should be exact for products of the reference basis functions and polynomials of degree $p-2$. Until now, the condition used for the accuracy of the quadrature rule was
\begin{align}
\label{eq:quadCOld}
\int_{\tilde e} \tilde f\;d\tilde x&=  \sum_{\tilde{\vx}\in\tilde{\mathcal{Q}}} \tilde{\omega}_{\tilde{\vx}}\tilde f(\vx), &&\text{for all }\tilde f\in\mathcal{P}_{p+p'-2}(\tilde e),
\end{align}
see for example \cite{cohen95, mulder96, chin99}, so it was imposed that the reference quadrature rule should be exact for functions in $\mathcal{P}_{p+p'-2}(\tilde e)$, with $p'$ the highest polynomial degree of the enriched space, which turns out to be significantly more restrictive for tetrahedral elements. By using (\ref{eq:quadC1}) instead of (\ref{eq:quadCOld}) we are able to develop new mass-lumped elements that require significantly less nodes.

In the next subsections we will prove that the convergence rate of the mass-lumped finite element method remains optimal under the less severe condition (\ref{eq:quadC1}). The novel part of the proofs are the bounds on the integration error, derived in Section \ref{sec:mainResult2}. This is the only part where we explicitly use condition (\ref{eq:quadC1}). Using these bounds we can prove optimal convergence in a rather standard way.

\subsection{Some Norms and Interpolation Properties}
For the convergence analysis, we use multiple interpolation properties, which we will present in this subsection. Also, to make the analysis more readable, we will use $C$ to denote some positive constant that may depend on the regularity of the mesh, the reference space $\tilde U$, the reference quadrature rule, the domain $\Omega$, and the parameters $\rho,c$, but does not depend on the mesh resolution $h$, the time interval $(0,T)$, or the choice of the functions that appear in the inequality.

Let $H^k(\Omega)$, with $k\geq 1$, denote the Sobolev space, consisting of functions with square integrable order-$k$ weak derivatives equipped with norm
\begin{align*}
\| u\|_k^2 &:= \sum_{|\valpha|\leq k} \|D^{\valpha} u\|_0^2, &&k\geq 1,
\end{align*}
where $\|\cdot\|_0$ denotes the standard $L^2(\Omega)$-norm, and $D^{\valpha}:=\partial_1^{\alpha_1}\partial_2^{\alpha_2}\partial_3^{\alpha_3}$ denotes a higher-order partial derivative of order $|\valpha|:=\alpha_1+\alpha_2+\alpha_3$. Also let $H^k(\mathcal{T}_h)$, with $k\geq 1$, denote the broken Sobolev space, consisting of functions that belong to $H^k(e)$ when restricted to element $e$, for all $e\in\mathcal{T}_h$. We equip this space with the norm
\begin{align*}
\| u\|_{\mathcal{T}_h,k}^2 &:=  \sum_{e\in\mathcal{T}_h} \|u\|_{e,k}^2 := \sum_{e\in\mathcal{T}_h}\left(\sum_{|\valpha|\leq k} \|D^{\valpha}\vu\|_e^2\right), &&k\geq 1.
\end{align*}
Now let $I_h:\mathcal{C}^0(\overline\Omega)\rightarrow U(\mathcal{T}_h,\tilde U)$ denote the interpolation by a function in $U(\mathcal{T}_h,\tilde U)$ through the nodes of $\mathcal{Q}_h$. This interpolation operator is well-defined for functions in $H^2(\mathcal{T}_h)\cap H^1(\Omega)$, since $H^2(e) \subset \mathcal{C}^0(\overline e)$ when $e$ is a three-dimensional element, and therefore $H^2(\mathcal{T}_h)\cap H^1(\Omega)\subset\mathcal{C}^0(\overline\Omega)$. For this interpolation operator, we can present the following approximation properties:
\begin{lem}
\label{lem:interP1}
Let $p\geq 2$ be the degree of the finite element space and let $u\in  H^1(\Omega)\cap H^k(\mathcal{T}_h)$ with $k\geq 2$. Then
\begin{align*}
\|u-I_hu\|_{\mathcal{T}_h,l} &\leq Ch^{\mathrm{min}(p+1,k)-l}\|u\|_{\mathcal{T}_h,\mathrm{min}(p+1,k)}, && l\leq\mathrm{min}(p+1,k).
\end{align*}
\end{lem}
\begin{proof}
This result follows from \cite[Theorem 3.1.6]{ciarlet78}.
\end{proof}

Now assume that the weights for the reference quadrature rule are all strictly positive. For any function in $H^1(\Omega)\cap H^2(\mathcal{T}_h)$, we can then define the following discrete $L^2$ semi-norm:
\begin{align*}
|u|_{\mathcal{Q}_h}^2 &:= (u,u)_{\mathcal{Q}_h}.
\end{align*}
This discrete semi-norm is well defined, since $H^1(\Omega)\cap H^2(\mathcal{T}_h)\subset\mathcal{C}^0(\overline\Omega)$ as mentioned before. This becomes a full norm, $\|\cdot\|_{\mathcal{Q}_h}$, that is equivalent to the $L^2$-norm, for functions in $U_h$:
\begin{lem}
\label{lem:MLnormP}
If all the weights of the reference quadrature rule are strictly positive, then
\begin{align}
C^{-1}\|u\|_0 &\leq \|u\|_{\mathcal{Q}_h} \leq C\|u\|_{0} &&\text{for all }u\in U_h.
\label{eq:MLnormP}
\end{align}
\end{lem}

\begin{proof}
Since the function space of the reference element $\tilde U:=\mathrm{span}\{\tilde w_{\tilde\vx}\}_{\tilde\vx\in\tilde{\mathcal{Q}}}$ is finite-dimensional, and since all weights of the reference quadrature rule $\{\tilde\omega_{\tilde\vx}\}_{\tilde\vx\in\tilde{\mathcal{Q}}}$ are positive, there exists a constant $C>0$ depending on the reference quadrature rule and function space $\tilde U$, such that
\begin{align*}
C^{-1}\|u\|_{\tilde e} &\leq \|\tilde u\|_{\tilde{\mathcal{Q}}} \leq C\|\tilde u\|_{\tilde e} &&\text{for all }\tilde u\in \tilde U.
\end{align*}
where $\|\tilde u\|_{\tilde{\mathcal{Q}}}^2 := (\tilde u,\tilde u)_{\tilde{\mathcal{Q}}}$. Then (\ref{eq:MLnormP}) follows from the relations
\begin{align*}
\|u\|_0^2 = \sum_{e\in\mathcal{T}_h} \frac{|e|}{|\tilde e|}\| \tilde u_e \|_{\tilde e}^2, \qquad
\|u\|_{\mathcal{Q}_h}^2 = \sum_{e\in\mathcal{T}_h} \frac{|e|}{|\tilde e|}\| \tilde u_e \|_{\tilde{\mathcal{Q}}}^2,
\end{align*}
where $\tilde u_e:=u\circ\phi_e$.
\end{proof}

Now let $\Pi_{h,q}:L^2(\Omega)\rightarrow\mathcal{P}_q(\mathcal{T}_h)$ denote the $L^2$-projection onto the space of piecewise \emph{nonconforming} polynomials of at most degree $q$: 
\begin{align*}
\mathcal{P}_q(\mathcal{T}_h):=\{u\in L^2(\Omega) \;|\; u|_e\in\mathcal{P}_q(e) \text{ for all }e\in\mathcal{T}_h\}.
\end{align*}
We then present the following interpolation properties:
\begin{lem}
\label{lem:interP2}
Let  $u\in H^k(\mathcal{T}_h)$ with $k\geq 2$, and let $q\geq 0$. Then
\begin{align}
\label{eq:interP2a}
\|u-\Pi_{h,q}u\|_{0} &\leq Ch^{\mathrm{min}(q+1,k)}\|u\|_{\mathcal{T}_h,\mathrm{min}(q+1,k)}.
\end{align}
Furthermore, if also $u\in H^1(\Omega)$, if $p\geq \max(q,2)$ is the degree of the finite element space, and if all the weights of the reference quadrature rule are strictly positive, then
\begin{align}
\label{eq:interP2b}
|u-\Pi_{h,q}u|_{\mathcal{Q}_h} &\leq Ch^{\mathrm{min}(q+1,k)}\|u\|_{\mathcal{T}_h,\mathrm{min}(q+1,k)}.
\end{align}
\end{lem}
\begin{proof}
The first inequality, (\ref{eq:interP2a}), follows from \cite[Theorem 3.1.6]{ciarlet78}. The second inequality can be derived as follows: 
\begin{align*}
|u-\Pi_{h,q}u|_{\mathcal{Q}_h} &=  |I_hu-\Pi_{h,q}u|_{\mathcal{Q}_h} \\
&\leq C \|I_hu-\Pi_{h,q}u\|_{0} \\
&\leq C(\|I_hu-u\|_0 + \|u-\Pi_{h,q}u\|_0) \\
&\leq  Ch^{\mathrm{min}(q+1,k)}\|u\|_{\mathcal{T}_h,\mathrm{min}(q+1,k)}
\end{align*}
where we used Lemma \ref{lem:MLnormP} in the second line, the triangle inequality in the third line, and Lemma \ref{lem:interP1} and (\ref{eq:interP2a}) in the last line.
\end{proof}

\subsection{Bounds on the Integration Error}
\label{sec:mainResult2}
In this section we will derive some useful bounds on the error of the quadrature rules for the mass matrix. The proofs of these bounds will be the only cases where we explicitly use the accuracy condition of the quadrature rule, given in (\ref{eq:quadC1}). Using these results we can prove optimal order of convergence of the mass-lumped finite element method in a rather standard way.

Let $u,w\in H^2(\mathcal{T}_h)$, and let $r_h(u,w):=(u,w)-(u,w)_{\mathcal{Q}_h}$ be the integration error of the mass matrix. We can derive the following bounds on $r_h$:
\begin{lem}
\label{lem:intP}
Let $p\geq 2$ be the degree of the finite element space, $u\in H^{k}(\Omega)$ with $k\geq 2$, and $w\in U_h$. If the reference quadrature rule satisfies (\ref{eq:quadC1}) and if all its weights are strictly positive, then
\begin{align}
\label{eq:intPa}
|r_h(u,w)| \leq Ch^{\min(p,k)}\|u\|_{\min(p,k)} \|w\|_{1}.
\end{align}
and
\begin{align}
\label{eq:intPb}
|r_h(u,w)| \leq Ch^{\min(p+1,k)}\|u\|_{\min(p+1,k)} \|w\|_{\mathcal{T}_h,2}.
\end{align}
\end{lem}

\begin{proof}
Using (\ref{eq:quadC1}) and the fact that $\mathcal{P}_{p-2}(\tilde e)\otimes\tilde U\supset \mathcal{P}_{p}(\tilde e)$ for $p\geq 2$, we can write
\begin{align*} 
r_h(u,w) &= r_h\big((u-\Pi_{h,p-1}u) + (\Pi_{h,p-1}u-\Pi_{h,p-2}u) + \Pi_{h,p-2}u,  \\
 &\phantom{=} \qquad(w-\Pi_{h,0}w) + \Pi_{h,0}w  \big) \\
 &= r_h(u-\Pi_{h,p-1}u, w) +  r_h(\Pi_{h,p-1}u-\Pi_{h,p-2}u, w-\Pi_{h,0}w).
\end{align*}
From this, the Cauchy--Schwarz inequality, and Lemma \ref{lem:interP2}, we can then obtain (\ref{eq:intPa}). 

Using (\ref{eq:quadC1}), we can also write
\begin{align*} 
r_h(u,w) &= r_h\Big(\big[(u-\Pi_{h,p}u) + (\Pi_{h,p}u-\Pi_{h,p-1}u) + (\Pi_{h,p-1}u-\Pi_{h,p-2}u) + \\   
 &\phantom{=} \qquad \Pi_{h,p-2}u\big], \big[(w-\Pi_{h,1}w) + (\Pi_{h,1}w - \Pi_{h,0}w)+ \Pi_{h,0}w\big] \Big) \\
 &= r_h(u-\Pi_{h,p}u, w) + r_h(\Pi_{h,p}u-\Pi_{h,p-1}u, w-\Pi_{h,0}w) + \\
 &\phantom{=}\qquad r_h(\Pi_{h,p-1}u-\Pi_{h,p-2}u, w-\Pi_{h,1}w).
\end{align*}
From this, the Cauchy--Schwarz inequality, and Lemma \ref{lem:interP2}, we can then obtain (\ref{eq:intPb}). 
\end{proof}

\subsection{Optimal Convergence for a Related Elliptic Problem}
To prove optimal convergence of the mass-lumped finite element method, we first prove optimal convergence for a related elliptic problem.

Let $v\in H^2(\mathcal{T}_h)$. The elliptic problem related to (\ref{eq:hypPWF}), is finding $u\in H_0^1(\Omega)$ such that
\begin{align}
\label{eq:ellPWF}
a(u,w) &= (v,w) &&\text{for all }w\in H_0^1(\Omega).
\end{align}
This problem is well defined since $a$ is coercive and bounded with respect to the $H^1(\Omega)$-norm, which follows from the boundedness of $c$ and Poincar\'e's inequality. 

The related mass-lumped method for solving this problem is finding $u_h\in U_h$ such that
\begin{align}
\label{eq:ellPML}
a(u_h,w) &= (v,w)_{\mathcal{Q}_h} &&\text{for all }w\in U_h.
\end{align}

In the next theorems we prove optimal convergence of this method in the $H^1$-norm and $L^2$-norm.

\begin{thm}[Optimal Convergence in the $H^1$-norm]
\label{thm:ellP1}
Let $u$ be the solution of (\ref{eq:ellPWF}) and $u_h$ the solution of (\ref{eq:ellPML}), with $p\geq 2$ the degree of the finite element space. Also, let $k_u,k_v\geq 2$, $u\in H^{k_u}(\Omega)$, and $v\in H^{k_v}(\Omega)$. If the reference quadrature rule satisfies (\ref{eq:quadC1}) and if all its weights are strictly positive, then
\begin{align}
\label{eq:ellB1}
\|u-u_h\|_1 \leq Ch^{\mathrm{min}(p,k_u-1,k_v)} (\| u \|_{\min(p+1,k_u)} + \| v \|_{\mathrm{min}(p,k_v)}).
\end{align}
\end{thm}

\begin{proof}
By definition of $u$ and $u_h$, we have
\begin{align*}
a(u-u_h,w) &= r_h(v,w), &&\text{for all }w\in U_h.
\end{align*}
By choosing $w=I_hu-u_h$ we can then obtain
\begin{align}
\label{eq:ell1a}
a(I_hu-u_h,I_hu-u_h) &= -a(u-I_h u,I_hu-u_h) + r_h(v,I_hu-u_h),
\end{align}
From the coercivity of $a$ it follows that
\begin{align}
\label{eq:ell1b}
\|I_hu-u_h\|_1^2 \leq C a(I_hu-u_h,I_hu-u_h).
\end{align}
From the boundedness of $a$ and Lemma \ref{lem:interP1} it follows that
\begin{align}
\label{eq:ell1c}
|a( u-I_h u,I_hu-u_h)| \leq Ch^{\min(p,k_u-1)}  \|u\|_{\min(p+1,k_u)}\|I_hu-u_h\|_1.
\end{align}
Using Lemma \ref{lem:intP} we obtain
\begin{align}
\label{eq:ell1d}
|r_h(v,I_hu-u_h)| \leq Ch^{\min(p,k_v)}\|v\|_{\min(p,k_v)} \|I_hu-u_h\|_1.
\end{align}
Combining (\ref{eq:ell1a}), (\ref{eq:ell1b}), (\ref{eq:ell1c}), and (\ref{eq:ell1d}) then gives
\begin{align}
\label{eq:ell1e}
\|I_hu-u_h\|_1 \leq Ch^{\mathrm{min}(p,k_u-1,k_v)} (\| u \|_{\min(p+1,k_u)} + \| v \|_{\mathrm{min}(p,k_v)}).
\end{align}
From Lemma \ref{lem:interP1} it also follows that
\begin{align}
\label{eq:ell1f}
\|u-I_hu\|_1 \leq Ch^{\min(p,k_u-1)} \|u\|_{\min(p+1,k_u)}.
\end{align}
Combining (\ref{eq:ell1e}) and (\ref{eq:ell1f}) then results in (\ref{eq:ellB1}).
\end{proof}

To prove optimal convergence in the $L^2$-norm, we make the following regularity assumption: for any $v\in L^2(\Omega)$, the solution $u$ of (\ref{eq:ellPWF}) is in $H^2(\Omega)$ and satisfies
\begin{align}
\label{eq:regA}
\|u\|_{2} \leq C\|v\|_0.
\end{align}
This is certainly true if $\partial\Omega$ is $\mathcal{C}^2$ and $c\in\mathcal{C}^1(\overline\Omega)$. 

\begin{thm}[Optimal Convergence in the $L^2$-norm]
\label{thm:ellP2}
Let $u$ be the solution of (\ref{eq:ellPWF}) and $u_h$ the solution of (\ref{eq:ellPML}), with $p\geq 2$ the degree of the finite element space. Also, let $k_u,k_v\geq 2$, $u\in H^{k_u}(\Omega)$, $v\in H^{k_v}(\Omega)$, and assume that the regularity condition (\ref{eq:regA}) holds. If the reference quadrature rule satisfies (\ref{eq:quadC1}) and if all its weights are strictly positive, then
\begin{align}
\label{eq:ellB2a}
\|u-u_h\|_0 \leq Ch^{\mathrm{min}(p+1,k_u,k_v)} (\| u \|_{\min(p+1,k_u)} + \| v \|_{\mathrm{min}(p+1,k_v)})
\end{align}
and
\begin{align}
\label{eq:ellB2b}
|u-u_h|_{\mathcal{Q}_h} \leq Ch^{\mathrm{min}(p+1,k_u,k_v)} (\| u \|_{\min(p+1,k_u)} + \| v \|_{\mathrm{min}(p+1,k_v)}).
\end{align}
\end{thm}

\begin{proof}
Let $z\in H^1_0(\Omega)$ be the solution of
\begin{align*}
a(z,w) &= (u-u_h,w), &&\text{for all }w\in H_0^1(\Omega).
\end{align*}
From the regularity assumption it follows that $z\in H^2(\Omega)$ and $\|z\|_{2}\leq C\|u-u_h\|_0$. Using the definition of $z$, $u$, and $u_h$, we can also write
\begin{align}
\|u-u_h\|_0^2 &= a(u-u_h,z) \nonumber \\
&= a(u-u_h,z-I_hz) + a(u-u_h,I_hz) \nonumber \\
&= a(u-u_h,z-I_hz) + r_h(v,I_hz). 
\label{eq:ell2a}
\end{align}
Using the boundedness of $a$, Theorem \ref{thm:ellP1}, Lemma \ref{lem:interP1}, and the regularity assumption, we obtain
\begin{align}
 |a(u-u_h,z-I_hz) |  
& \leq C\|u-u_h\|_1 \|z-I_hz\|_1 \nonumber \\
& \leq Ch^{\min(p+1,k_u,k_v+1)} (\|u\|_{\min(p+1,k_u)} + \|v\|_{\min(p,k_v)})\|u-u_h\|_0.
\label{eq:ell2c}
\end{align}
From Lemma \ref{lem:intP}, Lemma \ref{lem:interP1}, and the regularity assumption it also follows that
\begin{align}
\label{eq:ell2d}
|r_h(v,I_hz)| \leq Ch^{\min(p+1,k_v)}\|v\|_{\min(p+1,k_v)}\|u-u_h\|_{0}.
\end{align}
Combining (\ref{eq:ell2a}), (\ref{eq:ell2c}), and (\ref{eq:ell2d}) results in (\ref{eq:ellB2a}).

To derive (\ref{eq:ellB2b}), we use Lemma \ref{lem:MLnormP} to obtain
\begin{align*}
|u-u_h|_{\mathcal{Q}_h}&= \| I_hu-u_h\|_{\mathcal{Q}_h} \leq C\|I_hu-u_h\|_0.
\end{align*}
Combining this inequality with (\ref{eq:ellB2a}) and Lemma \ref{lem:interP1} results in (\ref{eq:ellB2b}).
\end{proof}


\subsection{Some Additional Norms and Interpolation Properties}
In order to analyze the convergence for the time dependent problem, we need to introduce an additional projection operator and some additional function spaces.

Let $L$ denote the spatial operator $L:=-\nabla\cdot c\nabla$, and let $u\in H_0^1(\Omega)$ with $Lu\in \mathcal{C}^0(\overline\Omega)$. We define the projection $\pi_hu\in U_h$ to be the solution of
\begin{align*}
a(\pi_hu,w) &= (Lu,w)_{\mathcal{Q}_h}, &&\text{for all }w\in U_h.
\end{align*}
We can derive the following interpolation property of this projection operator:
\begin{lem}
\label{lem:interP3}
Let $p\geq 2$ be the degree of the finite element space, and let $c\in\mathcal{C}^{k+1}(\overline\Omega)$ and $u\in H^1_0(\Omega)\cap H^{k+2}(\Omega)$, with $k\geq 2$. If the reference quadrature rule satisfies (\ref{eq:quadC1}) and if all its weights are strictly positive, then
\begin{align*}
\|u-\pi_hu\|_1 &\leq Ch^{\mathrm{min}(p,k)} \| u \|_{\min(p+2,k+2)} ,
\end{align*}
Moreover, if regularity condition (\ref{eq:regA}) also holds, then
\begin{align*}
\|u-\pi_hu\|_0 &\leq Ch^{\mathrm{min}(p+1,k)} \| u \|_{\min(p+3,k+2)}, \\
|u-\pi_hu|_{\mathcal{Q}_h} &\leq Ch^{\mathrm{min}(p+1,k)} \| u \|_{\min(p+3,k+2)} .
\end{align*}
\end{lem}

\begin{proof}
From partial integration it follows that $a(u,w)=(Lu,w)$ for all $w\in H_0^1(\Omega)$. Also, by definition of the projection we have $a(\pi_hu,w)=(Lu,w)_{\mathcal{Q}_h}$ for all $w\in U_h$. The inequalities then follow from Theorem \ref{thm:ellP1} and Theorem \ref{thm:ellP2} by taking $v=Lu$, $k_u=k+2$, $k_v=k$, and using the bounds $\|Lu\|_q\leq C\|u\|_{q+2}$ for $q\leq k$.
\end{proof}

We also extend the Sobolev spaces $H^k(\Omega)$ to Bochner spaces $L^{\infty}(0,T;H^k(\Omega))$, equipped with norm
\begin{align*}
\|u\|_{\infty,k} &:= \mathrm{ess}\sup_{t\in(0,T)} \|u\|_{k}.
\end{align*}

\subsection{Optimal Convergence of the Mass-Lumped Element Method}
In this section we prove the optimal convergence of the mass-lumped finite element method for the wave equation. We first derive an equation for the behavior of the numerical error and then prove optimal convergence in the energy-norm and $L^2$-norm.

\begin{lem}[Error Equation]
\label{lem:hypP1}
Let $u$ be the solution of (\ref{eq:hypPWF}) and let $u_h$ be the solution of (\ref{eq:MLFEM}). If $\rho\in\mathcal{C}^0(\overline\Omega)$,
$\partial_t^2u \in L^{2}(0,T; \mathcal{C}^0_0(\Omega))$, and $f \in L^2(0,T; \mathcal{C}^0(\overline\Omega))$, then $Lu\in L^{2}(0,T; \mathcal{C}^0(\overline\Omega))$, and 
\begin{align}
\label{eq:hypP1}
(\rho\partial_t^2e_h,w)_{\mathcal{Q}_h}+ a(e_h,w) &= -(\rho\partial_t^2\epsilon_h, w)_{\mathcal{Q}_h}
\end{align}
for all $w\in U_h$ and almost every $t\in(0,T)$, where $e_h:=\pi_hu-u_h$ and $\epsilon_h:=u-\pi_hu$.
\end{lem}

\begin{proof}
Since $\rho$ is bounded and continuous, it follows that $\rho\partial_t^2u\in L^{2}(0,T;\mathcal{C}^0_0(\Omega))$. Since also $f\in L^2(0,T; \mathcal{C}^0(\overline\Omega))$, it follows that $Lu\in L^{2}(0,T; \mathcal{C}^0(\overline\Omega))$ and $\rho\partial_t^2u + Lu = f$. This implies
\begin{align*}
(\rho\partial_t^2u,w)_{\mathcal{Q}_h} + (Lu,w)_{\mathcal{Q}_h} &= (f,w)_{\mathcal{Q}_h}
\end{align*}
for all $w\in U_h$ and almost every $t\in(0,T)$. Using the definition of $\pi_hu$ we can then obtain
\begin{align*}
(\rho\partial_t^2u,w)_{\mathcal{Q}_h} + a(\pi_h u,w) &= (f,w)_{\mathcal{Q}_h}
\end{align*}
for all $w\in U_h$ and almost every $t\in(0,T)$. By definition of $u_h$ we have
\begin{align*}
(\rho\partial_t^2u_h,w)_{\mathcal{Q}_h} + a(u_h,w) &= (f,w)_{\mathcal{Q}_h}
\end{align*} 
for all $w\in U_h$ and almost every $t\in(0,T)$. Subtracting this from the previous equality and reordering the terms results in (\ref{eq:hypP1}).
\end{proof}

\begin{thm}[Optimal Convergence in the Energy-Norm]
\label{lthm:hypP2}
Let $u$ be the solution of (\ref{eq:hypPWF}) and let $u_h$ be the solution of (\ref{eq:MLFEM}), with $p\geq 2$ the degree of the finite element space. Let $\rho\in\mathcal{C}(\overline\Omega)$, $f\in L^{2}(0,T; \mathcal{C}^0(\overline\Omega))$, and let $c\in\mathcal{C}^{k+1}(\overline\Omega)$, $u,\partial_tu,\partial_t^2u \in L^{\infty}(0,T; H^{k+2}(\Omega))$ for some $k\geq 2$. Also, assume that regularity condition (\ref{eq:regA}) holds. If the reference quadrature rule satisfies (\ref{eq:quadC1}) and if all its weights are strictly positive, then
\begin{align}
\label{eq:hypP2}
&\qquad\|u-u_h\|_{\infty,1} +  \|\partial_tu-\partial_tu_h\|_{\infty,0} \leq  \\
& Ch^{\min(p,k)} \big( \|u\|_{\infty,\min(p+3,k+2)} + \|\partial_tu\|_{\infty,\min(p+3,k+2)} + T\|\partial_t^2u\|_{\infty,\min(p+3,k+2)} \big) \nonumber. 
\end{align}
\end{thm}

\begin{proof}
Define $e_h:=\pi_hu-u_h$ and $\epsilon_h:=u-\pi_hu$. From Lemma \ref{lem:hypP1}, it follows that
\begin{align}
\label{eq:hypP2a}
(\rho\partial_t^2e_h,w)_{\mathcal{Q}_h}+ a(e_h,w) &= -(\rho\partial_t^2\epsilon_h, w)_{\mathcal{Q}_h},
\end{align}
for all $w\in U_h$ and almost every $t\in(0,T)$. By substituting $w=\partial_te_h$ we can obtain
\begin{align}
\label{eq:hypP2b}
\partial_tE_h &= -(\rho\partial_t^2\epsilon_h, \partial_te_h)_{\mathcal{Q}_h},
\end{align}
for almost every $t\in(0,T)$, where $E_h:=\frac12(\rho \partial_te_h,\partial_te_h)_{\mathcal{Q}_h} + \frac12a(e_h,e_h)$ is the discrete energy. Fix $T'\in(0,T)$ and integrate (\ref{eq:hypP2b}) over $(0,T')$ to obtain
\begin{align}
\label{eq:hypP2c}
E_h|_{t=T'} = E_h|_{t=0} - \int_{0}^{T'} (\rho\partial_t^2\epsilon_h, \partial_te_h)_{\mathcal{Q}_h} \;dt .
\end{align}
Using the coercivity of $a$, the boundedness of $\rho$, and Lemma \ref{lem:MLnormP}, we can derive
\begin{align}
\label{eq:hypP2d}
\|e_h\|_{1} +  \|\partial_te_h\|_{0} &\leq CE_h^{1/2}, &&\text{a.e. }t\in(0,T).
\end{align}
From the Cauchy--Schwarz inequality, the bounds of $\rho$, Lemma \ref{lem:interP3}, and Lemma \ref{lem:MLnormP}, we can also obtain
\begin{align}
\label{eq:hypP2e}
&|(\rho\partial_t^2\epsilon_h, \partial_te_h)_{\mathcal{Q}_h}| \leq Ch^{\min(p+1,k)}\|\partial_t^2u\|_{\min(p+3,k+2)}\|\partial_te_h\|_0,
\end{align}
for almost every $t\in (0,T)$. Finally, we can use Lemma \ref{lem:interP1}, Lemma \ref{lem:interP3}, and the boundedness of $\rho$ and $a$ to obtain
\begin{align}
\label{eq:hypP2f}
& E_h^{1/2}|_{t=0} \leq Ch^{\min(p,k)}\big(\|u\|_{\infty,\min(p+3,k+2)} + \| \partial_tu\|_{\infty,\min(p+3,k+2)}\big).
\end{align}
By taking the supremum of (\ref{eq:hypP2c}) for all $T'\in(0,T)$ and using (\ref{eq:hypP2d}), (\ref{eq:hypP2e}), and (\ref{eq:hypP2f}), we can obtain
\begin{align}
\label{eq:hypP2h}
&\qquad \|e_h\|_{\infty,1} +  \|\partial_te_h\|_{\infty,0} \leq \\
&  Ch^{\min(p,k)} \big( \|u\|_{\infty,\min(p+3,k+2)} + \|\partial_tu\|_{\infty,\min(p+3,k+2)} + T\|\partial_t^2u\|_{\infty,\min(p+3,k+2)} \big) \nonumber.
\end{align}
Using (\ref{eq:hypP2h}) and Lemma \ref{lem:interP3} we obtain (\ref{eq:hypP2}).
\end{proof}


\begin{thm}[Optimal Convergence in the $L^2$-Norm]
\label{thm:hypP3}
Let $u$ be the solution of (\ref{eq:hypPWF}) and let $u_h$ be the solution of (\ref{eq:MLFEM}), with $p\geq 2$ the degree of the finite element space. Let $\rho\in\mathcal{C}(\overline\Omega)$, $f\in L^{2}(0,T; \mathcal{C}^0(\overline\Omega))$, $\partial_t^2u\in L^{2}(0,T; \mathcal{C}^0_0(\Omega))$, and let $c\in\mathcal{C}^{k+1}(\overline\Omega)$, $u,\partial_tu,\in L^{\infty}(0,T; H^{k+2}(\Omega))$ for some $k\geq 2$. Also, assume that regularity condition (\ref{eq:regA}) holds. If the reference quadrature rule satisfies (\ref{eq:quadC1}) and if all its weights are strictly positive, then 
\begin{align}
\label{eq:hypP3}
&\|u-u_h\|_{\infty,0} \leq Ch^{\min(p+1,k)} \big(\|u\|_{\infty,\min(p+3,k+2)}  + T\|\partial_tu\|_{\infty,\min(p+3,k+2)} \big).  
\end{align}
\end{thm}

\begin{proof}
Define $e_h:=\pi_hu-u_h$ and $\epsilon_h:=u-\pi_hu$. From Lemma \ref{lem:hypP1}, it follows that
\begin{align}
\label{eq:hypP3a}
(\rho\partial_t^2e_h,w)_{\mathcal{Q}_h} + a(e_h,w) &= -(\rho\partial_t^2\epsilon_h, w)_{\mathcal{Q}_h},
\end{align}
for all $w\in U_h$ and almost every $t\in(0,T)$. Fix $T'\in(0,T)$ and choose $w$ as
\begin{align*}
w|_{t=t'}&:= \int_{t'}^{T'} e_h \;dt.
\end{align*}
This implies that $w|_{t=T'}=0$ and $\partial_tw=-e_h$. Using the relations
\begin{align*}
(\rho\partial_t^2e_h,w)_{\mathcal{Q}_h}  &= \partial_t(\rho\partial_te_h,w)_{\mathcal{Q}_h} + \frac12\partial_t(\rho e_h,e_h)_{\mathcal{Q}_h}, \\
a(e_h,w) &= -\frac12\partial_ta(w,w), \\
-(\rho\partial_t^2\epsilon_h, w)_{\mathcal{Q}_h} &= -\partial_t(\rho\partial_t\epsilon_h, w)_{\mathcal{Q}_h} - (\rho\partial_t\epsilon_h, e_h)_{\mathcal{Q}_h},
\end{align*}
we can rewrite (\ref{eq:hypP3a}) as
\begin{align}
\label{eq:hypP3b}
\frac12\partial_t(\rho e_h,e_h)_{\mathcal{Q}_h} &= \frac12\partial_ta(w,w) -\partial_t\big(\rho\partial_t(u-u_h),w\big)_{\mathcal{Q}_h} - (\rho\partial_t\epsilon_h,e_h)_{\mathcal{Q}_h},
\end{align} 
for almost every $t\in(0,T)$. Integrating (\ref{eq:hypP3b}) over $(0,T')$ and using the fact that $w|_{t=T'}=0$ and $\partial_t(u-u_h)|_{t=0,x\in\mathcal{Q}_h}=0$, results in
\begin{align}
\label{eq:hypP3c}
& \frac12(\rho e_h,e_h)_{\mathcal{Q}_h}|_{t=T'} = \frac12(\rho e_h,e_h)_{\mathcal{Q}_h}|_{t=0} -\frac12a(w,w)|_{t=0} - \int_{0}^{T'} (\rho\partial_t\epsilon_h,e_h)_{\mathcal{Q}_h}\;dt. 
\end{align}
From the boundedness of $\rho$ and Lemma \ref{lem:MLnormP} it follows that
\begin{align}
\label{eq:hypP3d}
\|e_h\|_0 &\leq C\|\rho^{1/2} e_h\|_{\mathcal{Q}_h}, &&\text{a.e. }t\in(0,T).
\end{align}
Because of the coercivity of $a$ we have
\begin{align}
\label{eq:hypP3e}
 -\frac12a(w,w)|_{t=0} &< 0.
\end{align}
From the Cauchy--Schwarz inequality, the bounds of $\rho$, Lemma \ref{lem:interP3}, and Lemma \ref{lem:MLnormP}, we can also obtain
\begin{align}
\label{eq:hypP3f}
&|(\rho\partial_t\epsilon_h, e_h)_{\mathcal{Q}_h}| \leq Ch^{\min(p+1,k)}\|\partial_tu\|_{\min(p+3,k+2)} \|e_h\|_0,
\end{align}
for almost every $t\in(0,T)$. Finally, we can use Lemma \ref{lem:interP1}, Lemma \ref{lem:interP3} and the boundedness of $\rho$ to obtain
\begin{align}
\label{eq:hypP3g}
& \|\rho^{1/2} e_h|_{t=0}\|_{\mathcal{Q}_h} \leq Ch^{\min(p+1,k)}\|u\|_{\infty, \min(p+3,k+2)} . 
\end{align}
By taking the supremum of (\ref{eq:hypP3c}) for all $T'\in(0,T)$ and using (\ref{eq:hypP3d}), (\ref{eq:hypP3e}), (\ref{eq:hypP3f}), and (\ref{eq:hypP3g}), we can obtain
\begin{align}
\label{eq:hypP3i}
&\|e_h\|_{\infty,0} \leq Ch^{\min(p+1,k)} \big(\|u\|_{\infty,\min(p+3,k+2)}  + T\|\partial_tu\|_{\infty,\min(p+3,k+2)} \big).
\end{align}
Using (\ref{eq:hypP3i}) and Lemma \ref{lem:interP3} we obtain (\ref{eq:hypP3}).
\end{proof}


\section{Several New Mass-Lumped Tetrahedral Elements of Degrees Two to Four}
\label{sec:MLelements}

In this section, we present several novel mass-lumped tetrahedral elements for degree $p=2,3,4$. The new degree-2 and degree-3 elements use 15 and 32 nodes per element, respectively, while the current elements for these degrees require 23 and 50 nodes, respectively \cite{mulder96, chin99}. We also introduce several degree-4 elements, requiring 60, 61, and 65 nodes. Mass-lumped tetrahedral elements of degree four had not been found yet.

\begin{table}[h]
\caption{Degree-2 mass-lumped tetrahedral element with 15 nodes.}
\label{tab:ML2n15}
\begin{center}
{\tabulinesep=1.0mm
\begin{tabu}{c r c c}
Nodes					& $n$ 	& $\omega$ 			& parameters  \\ \hline
$\{(0,0,0)\}	$				& $4$	& $\frac{17}{5040}$		& - \\
$\{(\frac12,\frac12,0)\}$		& $6$	& $\frac{2}{315}$		& - \\
$\{(\frac13,\frac13,0)\}$		& $4$	& $\frac{9}{560}$		& - \\
$\{(\frac14,\frac14,\frac14)\}$	& $1$	& $\frac{16}{315}$		& - \\ \hline
 \multicolumn{4}{c}{$U=\mathcal{P}_2\oplus\mathcal{B}_f\oplus\mathcal{B}_e = \{x_1,x_1x_2,\beta_f,\beta_e\}$}  \\ \hline
\end{tabu}}
\end{center}
\end{table}

\begin{table}[h]
\caption{Degree-3 mass-lumped tetrahedral element with 32 nodes.}
\label{tab:ML3n32}
\begin{center}
{\tabulinesep=1.0mm
\begin{tabu}{c r c c}
Nodes			& $n$ 	& $\omega$ 			& parameters  \\ \hline
$\{(0,0,0)\}	$		& $4$	& $\frac{41-9\sqrt{2}}{41160}$		& - \\
$\{(a,0,0)\}$		& $12$	& $\frac{8+9\sqrt{2}}{13720}$		& $\frac{3-\sqrt{3(\sqrt{2}-1)}}{6}$ \\
$\{(b,b,0)\}$		& $12$	& $\frac{10-\sqrt{2}}{1715}$		& $\frac{4-\sqrt{2}}{12}$ \\
$\{(c,c,c)\}$		& $4$	& $\frac{3}{140}$				& $\frac16$ \\ \hline
 \multicolumn{4}{c}{$U=\mathcal{P}_3\oplus\mathcal{B}_f\mathcal{P}_1\oplus\mathcal{B}_e\mathcal{P}_1=\{x_1,x_1^2x_2,\beta_fx_1,\beta_ex_1\}$}  \\ \hline
 \multicolumn{4}{c}{$U\otimes\mathcal{P}_1=\{x_1, x_1^2x_2, x_1^2x_2^2, \beta_fx_1,\beta_fx_1x_2, \beta_ex_1,\beta_ex_1x_2\}$}  \\ \hline
\end{tabu}}
\end{center}
\end{table}

\begin{table}[h]
\caption{Degree-4 mass-lumped tetrahedral element with 65 nodes.}
\label{tab:ML4n65}
\begin{center}
{\tabulinesep=1.0mm
\begin{tabu}{c r l l}
Nodes					& $n$ 	& $\omega$ 	& parameters  \\ \hline
$\{(0,0,0)\}	$				& $4$	& $0.0001216042545112321$		& - \\
$\{(a,0,0)\}$				& $12$	& $0.0004704124198744411$		& $0.1724919407749086$ \\
$\{(\frac12,0,0)\}$			& $6$	& $0.0001767065925083475$		& - \\
$\{(b_1,b_1,0)\}$			& $12$	& $0.001974748586596177$		& $0.1474177969013686$ \\
$\{(b_2,b_2,0)\}$			& $12$	& $0.001192465311769701$		& $0.4540395272271067$ \\
$\{(\frac13,\frac13,0)\}$		& $4$	& $0.001044697597634123$		& - \\
$\{(c_1,c_1,c_1)\}$			& $4$	& $0.008841425190569096$		& $0.1282209316290979$ \\ 
$\{(d,d,\frac12-d)\}$			& $6$	& $0.006891012924401557$		& $0.08742182088664353$ \\ 
$\{(c_2,c_2,c_2)\}$			& $4$	& $0.007499563520517103$		& $0.3124061452070811$ \\ 
$\{(\frac14,\frac14,\frac14)\}$	& $1$	& $0.01057967149339721$		& - \\ \hline
 \multicolumn{4}{c}{$U=\mathcal{P}_4 \oplus \mathcal{B}_f(\mathcal{P}_2\oplus\mathcal{B}_f) \oplus \mathcal{B}_e(\mathcal{P}_2\oplus\mathcal{B}_f\oplus\mathcal{B}_e)$}  \\
 \multicolumn{4}{c}{$=\{x_1,x_1^2x_2,x_1^2x_2^2,\beta_fx_1, \beta_fx_1x_2,\beta_f^2, \beta_ex_1,\beta_ex_1x_2,\beta_e\beta_f,\beta_e^2\} $}  \\ \hline
 \multicolumn{4}{c}{$U\otimes\mathcal{P}_2=\{x_1, x_1^2x_2, x_1^3x_2^2, x_1^3x_2^3, \beta_fx_1, \beta_fx_1^2x_2, \beta_fx_1^2x_2^2,\beta_f^2x_1,\beta_f^2x_1x_2, \dots$}  \\  
  \multicolumn{4}{c}{$\dots, \beta_ex_1,\beta_ex_1^2x_2,\beta_ex_1^2x_2^2, \beta_e\beta_fx_1,\beta_e\beta_fx_1x_2, \beta_e^2x_1,\beta_e^2x_1x_2\}$}  \\  \hline
\end{tabu}}
\end{center}
\end{table}

We present the mass-lumped tetrahedral elements using the reference tetrahedron with vertices at $(0,0,0)$, $(1,0,0)$, $(0,1,0)$, and $(0,0,1)$. In previous sections we used a tilde to denote coordinates and sets in the reference space, but since we only consider the reference space in this section, we will drop the tilde for readability. 

The nodes on the reference element are described using the notation $\{\vx\}$, which denotes the node $\vx$ and all equivalent nodes $s(\vx)$, with $s\in\mathcal{S}$. As shown in Lemma \ref{lem:Srep}, any $s\in\mathcal{S}$ can be represented by a permutation of the barycentric coordinates. In this case, the barycentric coordinates are given by the three Cartesian coordinates $x_1$, $x_2$, $x_3$, and the additional coordinate $x_4:=1-x_1-x_2-x_3$, so any $s\in\mathcal{S}$ can be written as $s(x_1,x_2,x_3)=(x_j,x_j,x_k)$, with $i,j,k\in\{1,2,3,4\}$, $i\neq j$, $i\neq k$, $j\neq k$. The barycentric coordinates of the node $\vx=(\frac15,\frac15,\frac15)$, for example, are therefore given by $(\frac15,\frac15,\frac15,\frac25)$, and the set of equivalent nodes $\{\vx\}$ consists of $(\frac15,\frac15,\frac15)$, $(\frac25,\frac15,\frac15)$, $(\frac15,\frac25,\frac15)$, and $(\frac15,\frac15,\frac25)$.

The reference function space, denoted by $U$, is the span of all nodal basis functions and is described in terms of $\{w\}$, which denotes the span of function $w$ and all its equivalent functions $w\circ s$, with $s\in\mathcal{S}$. For example, all equivalent functions of $w=x_1x_2$ are $x_1x_2$, $x_1x_3$, $x_1x_4$, $x_2x_3$, $x_2x_4$, and $x_3x_4$, so $\{w\}$ is the span of these six functions.

We assign the same weight to each equivalent node, so $\omega_{\vx}=\omega_{s(\vx)}$ for all $s\in\mathcal{S}$. From this and properties (\ref{eq:confC1}) and (\ref{eq:confC2}) it follows that if the quadrature rule is exact for a function $w$, then it is exact for all equivalent functions in $\{w\}$. If we can describe a function space in the form of $\{w_1,w_2,\dots,w_N\}$, by which we mean the span of $w_1,w_2,\dots,w_N$ and all their equivalent versions, this means the quadrature rule is exact when it is exact for the $N$ functions $w_1,w_2,\dots,w_N$. 

To give an example, the degree-3 element, given in Table \ref{tab:ML3n32}, consists of the nodes $(0,0,0)$, $(a,0,0)$, $(b,b,0)$, $(c,c,c)$, and all equivalent nodes, and the function space for this element is given by $U=\mathcal{P}_3\oplus\mathcal{B}_f\mathcal{P}_1\oplus\mathcal{B}_e\mathcal{P}_1$, where $\mathcal{B}_f:=\{\beta_f\}:=\{x_1x_2x_3\}$ are the face bubble functions and $\mathcal{B}_e:=\{\beta_e\}:=\{x_1x_2x_3x_4\}$ is the internal bubble function and where we used the notation $\mathcal{B}_f\mathcal{P}_k:=\mathcal{B}_f\otimes\mathcal{P}_k$, $\mathcal{B}_e\mathcal{P}_k:=\mathcal{B}_{e}\otimes\mathcal{P}_k$. The quadrature rule should be exact for all functions in $U\otimes\mathcal{P}_1$, which can be written as
\begin{align*}
U\otimes\mathcal{P}_1 = \{x_1, x_1^2x_2, x_1^2x_2^2, \beta_fx_1,\beta_fx_1x_2, \beta_ex_1,\beta_ex_1x_2\},
\end{align*}
so as the span of 7 independent functions and all their equivalents. This means the quadrature rule should be exact for these 7 functions. Since this quadrature rule also has 7 parameters, namely 4 weights and three position parameters $a,b,c$, this results in a system of 7 equations with 7 unknowns. Solving this system results in the parameters given in Table \ref{tab:ML3n32}. 

This approach has also been used to obtain the other elements presented in this paper. We have not yet found a systematic way to determine a suitable function space $U\supset\mathcal{P}_p$ with a suitable configuration of the nodes. Instead, we just tried multiple configurations and checked if the resulting weights are all positive and the resulting nodes all lie on the reference triangle.

The degree-2 element with 15 nodes, the degree-3 element with 32 nodes, and the degree-4 element with 65 nodes are given in Tables \ref{tab:ML2n15}, \ref{tab:ML3n32}, and \ref{tab:ML4n65}, respectively. In these tables, $n$ denotes the number of nodes in the given equivalence class. Variants of the degree-4 element, requiring only 60 and 61 nodes, are given in Section \ref{sec:varML4}.

In the next sections we test these new mass-lumped elements and compare them with the current mass-lumped elements and several discontinuous Galerkin approximations.

\section{Dispersion Analysis}
\label{sec:dispersion}

In this section we analyze the dispersion properties of the mass-lumped elements. The dispersion error is measured by the difference between the propagation speed of physical and numerical waves and is one of the main criteria to judge the quality of the finite elements for wave propagation modelling. We will use it to obtain an indication of the required mesh resolution for a given accuracy, and to compare different finite element methods in terms of accuracy and numerical cost.

For the analysis we will follow the same procedure as in \cite{geevers18a}. We consider a homogeneous medium with $\rho, c=1$ and consider physical plane waves of the form 
\begin{align*}
u=e^{\im(\vkappa\cdot\vx-\omega t)},
\end{align*}
where $\im:=\sqrt{-1}$ is the imaginary number, $\vkappa$ is the wave vector, and $\omega$ is the angular velocity. Since $\rho,c=1$ we have a wave propagation speed $c_P=1$. For a given wave vector $\vkappa$ we compute all corresponding numerical plane waves and determine the numerical wave with a propagation speed $c_{P,h}=\omega_h/|\vkappa|$ closest to the physical wave velocity. The dispersion error is defined as the relative difference $(c_P-c_{P,h})/c_P$. We then find the worst case among all possible wave directions for a fixed wave length $\lambda=2\pi/|\vkappa|$. We determine the dispersion error for different wavelengths and extrapolate the results to obtain a relation between the dispersion error and number of elements per wave length.

\begin{figure}[h]
\centering
\begin{subfigure}[b]{0.45\textwidth}
  \includegraphics[width=\textwidth]{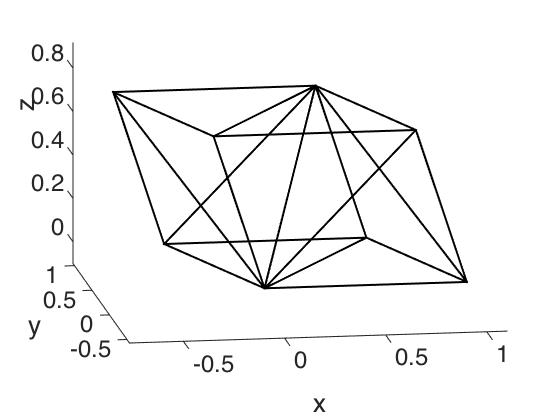}
\end{subfigure}
\begin{subfigure}[b]{0.45\textwidth}
  \includegraphics[width=\textwidth]{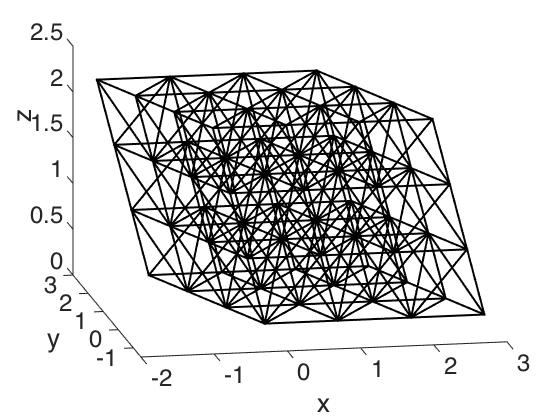}
\end{subfigure}
\caption{Single parallelepiped cell packed with tetrahedra (left), and a repeated pattern of these cells (right).}
\label{fig:tetra0}
\end{figure}

To obtain the numerical plane waves we construct a periodic tetrahedral mesh by packing a single parallelepiped cell with tetrahedra, and then repeating this pattern to fill the entire 3D-space. An illustration of such a mesh is given in Figure \ref{fig:tetra0}. Such a periodic mesh enables us to compute the numerical plane waves using Fourier modes and by solving an eigenvalue problem related to a single cell. 

To do this, let $\Omega_0$ be the parallelepiped cell at the origin. We can write $\Omega_0:=\ten{T}\cdot[0,1)^3=\{y\;|\;y=\ten{T}\cdot\vx, \text{for some }\vx\in[0,1)^3\}$, with $\ten{T}\in\mathbb{R}^{3\times 3}$ the second-order tensor whose columns are the vectors of the edges of $\Omega_0$ connected to the origin. Let $\{\vx^{(\Omega_0,i)}\}_{i=1}^{n_0}$ be the set of nodes on ${\Omega}_0$. For each $\vct{k}\in\mathbb{Z}^3$, we define the translated cell $\Omega_{\vct{k}}:=\ten{T}\cdot\vct{k}+\Omega_0$, and let $\{\vx^{(\Omega_{\vct{k}},i)}\}_{i=1}^{n_0}$ be the corresponding translated nodes. Then, for each node $\vx^{(\Omega_{\vct{k}},i)}$, we define $w^{(\Omega_{\vct{k}},i)}$ to be the corresponding nodal basis function. We can then define the following submatrices:
\begin{align*}
M^{(\Omega_0)}_{ij} &:= \left(\rho w^{(\Omega_{\vct{0}},i)},w^{(\Omega_{\vct{0}},j)}\right)_{\mathcal{Q}_h}, &&i,j=1,\dots,n_0, \\
A^{(\Omega_{0},\Omega_{\vct{k}})}_{ij} &:= a\left(w^{(\Omega_{\vct{0}},i)},w^{(\Omega_{\vct{k}},j)}\right), &&\vct{k}\in\{-1,0,1\}^3, \,i,j=1,\dots,n_0.
\end{align*}
For each wave vector $\vkappa$ we then define the matrix
\begin{align*}
S^{(\vkappa)} &:= M_{inv}^{(\Omega_0)}\left(\sum_{\vct{k}\in\{-1,0,1\}^3} e^{\im(\vkappa\cdot\ten{T}\cdot\vct{k})}A^{(\Omega_0,\Omega_\vct{k})}\right)
\end{align*}
where $M_{inv}^{(\Omega_0)}$ denotes the inverse of $M^{(\Omega_0)}$. For an order-$2K$ Dablain scheme, with time step size $\Delta t$, the angular frequencies of the numerical plane waves $\{\omega_{h}^{(\vkappa,i)}\}_{i=1}^{n_0}$ are given by
\begin{align*}
\omega_{h}^{(\vkappa,i)} &= \pm\frac{1}{\Delta t}\arccos\left(\sum_{k=0}^K\frac{1}{(2k)!}(-\Delta t^2s_h^{(\vkappa,i)})^k\right) ,
\end{align*}
where $\{s_h^{(\vkappa,i)}\}_{i=1}^{n_0}$ are the eigenvalues of $\sigma(S^{(\vkappa)})$ \cite{geevers18a}. The numerical wave propagation speed is given by $c_{P,h}^{(\vkappa,i)}=|\omega_{h}^{(\vkappa,i)}|/|\vkappa|$. The dispersion error, for a given wavelength $\lambda$, is then given by
\begin{align*}
e_{disp}(\lambda) := \sup_{\vkappa\in\mathbb{R}^3, |\vkappa|=2\pi/\lambda} \left( \inf_{i=1,\dots,n_0} \frac{|c_{P,h}^{(\vkappa,i)}-c_P|}{c_P} \right).
\end{align*}

For our dispersion analysis, we will consider a congruent, nearly-regular, equifacial mesh, known as the tetragonal disphenoid honeycomb. This mesh can be obtained by a repeated pattern of cells, where a single cell can be obtained by slicing the unit cube into six tetrahedra with the planes $x_1=x_2$, $x_1=x_3$, and $x_2=x_3$, and then applying the transformation $\vx\rightarrow \ten{T}\cdot\vx$, with 
\begin{align*}
\ten{T} := \begin{bmatrix}
1 & -1/3 & -1/3 \\
0 & \sqrt{8/9} & -\sqrt{2/9} \\
0 & 0 & \sqrt{2/3}
\end{bmatrix}.
\end{align*} 

We will analyze the relation between the dispersion error and the number of elements per wavelength $N_E:=(\lambda^3/|e|_{av})^{1/3}$, where $|e|_{av}=2\sqrt{3}/27$ denotes the average element volume. We will also look at the following quantities:
\begin{itemize}
  \item $n_{vec}=n_0\frac{\lambda^3}{|\Omega_0|}$: the number of degrees of freedom per $\lambda^3$-volume. Here $|\Omega_0|=4\sqrt{3}/9$ denotes the volume of $\Omega_0$.
  \item $n_{mat}=\frac{\lambda^3}{|\Omega_0|}\sum_{q\in\mathcal{Q}_{\Omega_0}} |\mathcal{N}(q)|$: the number of non-zero entries of the stiffness matrix per $\lambda^3$-volume. Here, $\mathcal{Q}_{\Omega_0}$ denotes the nodes on $\Omega_0$ and $|\mathcal{N}(q)|$ denotes the number of nodes connected with $q$ through an element.
  \item $N_{\Delta t}=T_0/\Delta t$: the number of time steps during one oscillation in time. Here $T_0=\lambda/c_P$ denotes the duration of one oscillation and $\Delta t=\sqrt{c_K/s_{h,max}}$ is the largest allowed time step size for the order-$2K$ Dablain scheme, with $c_K$ a constant depending on the order of the time integration scheme ($c_K=4,12,7.57,21.48$ for $K=1,2,3,4$, respectively) and 
  \begin{align*}
  s_{h,max} := \sup_{\vct{k}\in\mathcal{K}_0} \max_{i=1,\dots,n_0}s^{(\vkappa,i)}_{h}
  \end{align*}
the largest possible eigenvalue $s_h$, with $\mathcal{K}_0=\ten{T}^{-t}\cdot[0,2\pi)$ the space of distinct wave vectors. 
  \item $n_{comp}=n_{mat}KN_{\Delta t}$: the estimated number of computations per $\lambda^3$-volume during one time oscillation, with $K$ the number of stages of the order-$2K$ Dablain scheme.
\end{itemize}
Details on the dispersion analysis and how the quantities listed above are computed can be found in \cite{geevers18a}. 

\begin{figure}[h]
\centering
\includegraphics[width=0.7\textwidth]{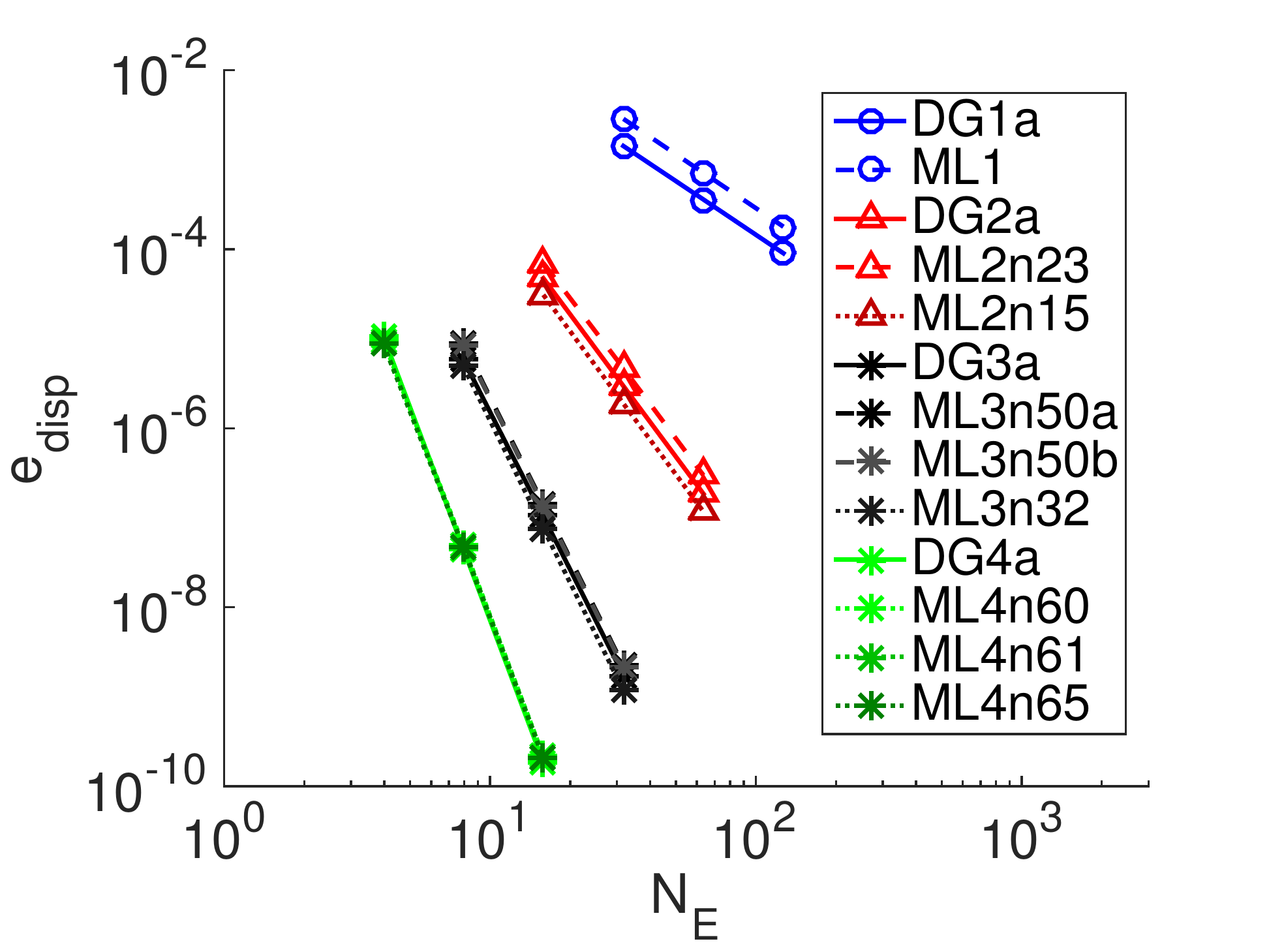}
\caption{Relation between the dispersion error and number of elements per wavelength for different mass-lumped finite element methods. The graphs of ML3n50a and ML3n50b, and the graphs of the degree-four methods are almost identical.}
\label{fig:errDisptAc}
\end{figure}

\begin{table}[h]
\caption{Approximation of the dispersion error. The new mass-lumped methods are marked in bold.}
\label{tab:errDisp}
\begin{center}
\begin{tabular}{c|c||c|c}
Method			& $e_{disp}$ 				& Method			& $e_{disp}$ \\ \hline\hline
DG1				& $1.45(n_{E})^{-2}$ 	& DG2 			& $3.00(n_{E})^{-4}$ \\
ML1				& $2.87(n_{E})^{-2}$ 	& ML2n23			& $4.82(n_{E})^{-4}$ \\ 
  				& 					 & \textbf{ML2n15} 	& $1.89(n_{E})^{-4}$ \\ \hline
DG3 				& $1.77(n_{E})^{-6}$ 	& DG4 			& $0.739(n_{E})^{-8}$\\ 
ML3n50a 			& $2.25(n_{E})^{-6}$ 	& \textbf{ML4n60}  	& $0.865(n_{E})^{-8}$ \\ 
ML3n50b 			& $2.15(n_{E})^{-6}$ 	& \textbf{ML4n61}  	& $0.854(n_{E})^{-8}$ \\
\textbf{ML3n32}  	& $1.19(n_{E})^{-6}$ 	& \textbf{ML4n65}  	& $0.825(n_{E})^{-8}$ \\
\end{tabular}
\end{center}
\end{table}

\begin{figure}[h]
\centering
\includegraphics[width=0.7\textwidth]{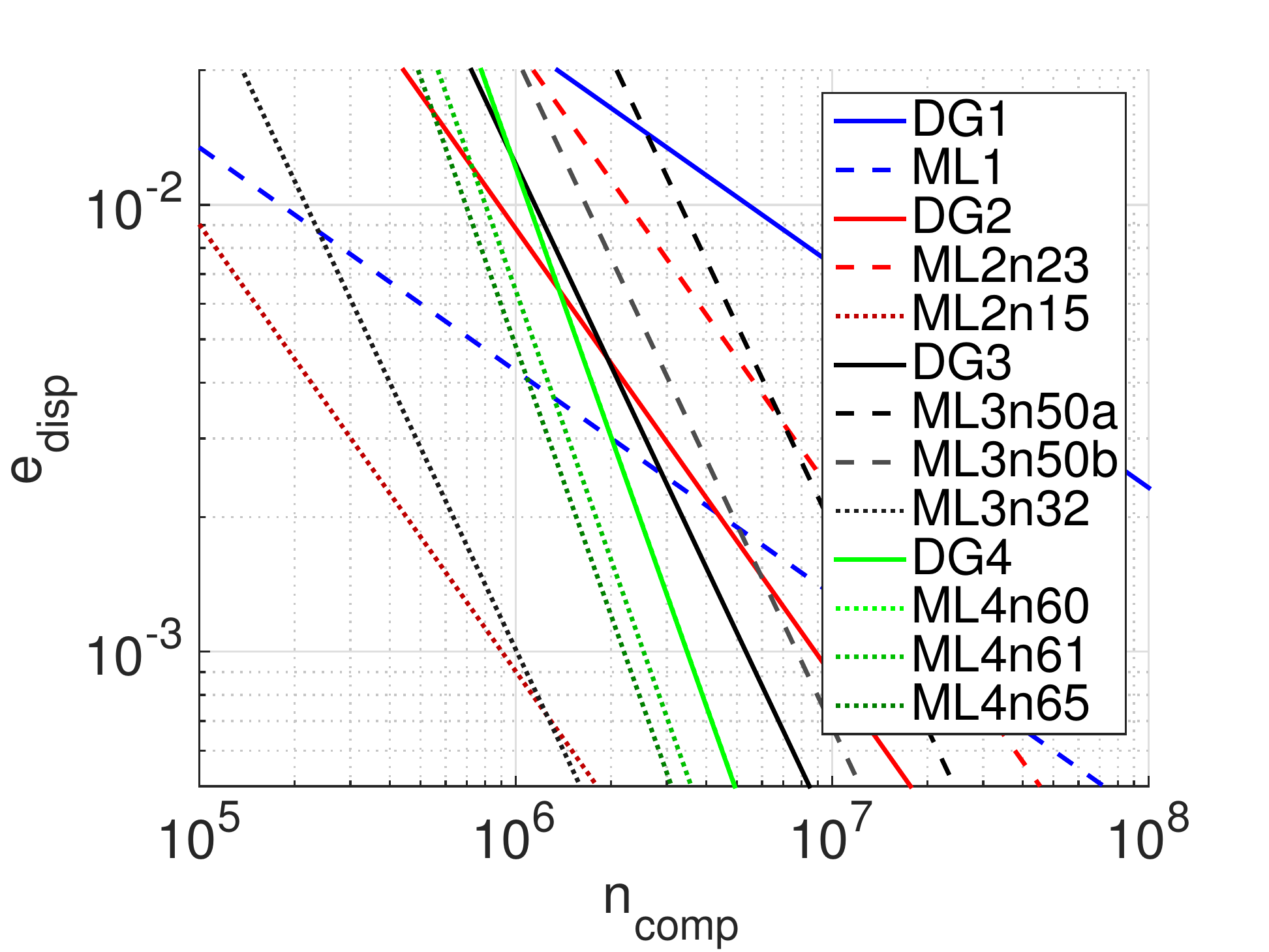}
\caption{Relation between the dispersion error and estimated computational cost for different finite element methods. The new mass-lumped methods are illustrated with dotted lines. The graphs of DG4 and ML4n60 are almost identical.}
\label{fig:errDisptAcComp}
\end{figure}

\begin{table}[h]
\begin{center}
\begin{tabular}{l|| r|r|r|r|r}
  & \multicolumn{5}{c}{$e_{disp}=0.001$}  \\ \hline
Method		& $N_E$ 	& $n_{vec}$ & $n_{mat}$ 		& $N_{\Delta t}$ & $n_{comp}$  	\\ \hline\hline
DG1	 		& $38.0$  		& $220000$	& $4400\times 10^3$& $120$ 	& $540.00\times 10^6$ 	  \\
ML1 			& $54.0$  		& $26000$	& $390\times 10^3$ 	& $47$ 	& $18.00\times 10^6$  	 \\ \hline
DG2  		& $7.4$ 		& $4100$  	& $200\times 10^3$ 	& $22$ 	& $8.80\times 10^6$ 	 \\
ML2n23  		& $8.3$ 		& $4800$ 		& $220\times 10^3$ 	& $52$ 	& $23.00\times 10^6$ 	\\ 
\textbf{ML2n15}	& $6.6$		& $1200$ 		& $39\times 10^3$ 	& $11$ 	& $0.90\times 10^6$   	\\ \hline
DG3  		& $3.5$ 		& $840$ 		& $84\times 10^3$ 	& $21$ 	& $5.30\times 10^6$ 	 \\
ML3n50a 		& $3.6$ 		& $1200$ 		& $98\times 10^3$ 	& $52$ 	& $15.00\times 10^6$ 	 \\
ML3n50b 		& $3.6$ 		& $1100$ 		& $96\times 10^3$ 	& $27$ 	& $7.70\times 10^6$ 	 \\
\textbf{ML3n32}& $3.2$		& $430$ 		& $26\times 10^3$ 	& $13$ 	& $1.00\times 10^6$ 	 \\ \hline
DG4  	  	& $2.3$ 		& $420$ 		& $73\times 10^3$ 	& $12$ 	& $3.50\times 10^6$ 	 \\
\textbf{ML4n60}& $2.3$ 		& $370$ 		& $38\times 10^3$ 	& $23$ 	& $3.50\times 10^6$ 	 \\
\textbf{ML4n61}& $2.3$		& $390$ 		& $39\times 10^3$ 	& $16$ 	& $2.50\times 10^6$ 	\\
\textbf{ML4n65}& $2.3$		& $410$ 		& $44\times 10^3$ 	& $13$ 	& $2.20\times 10^6$ 	 \\
\end{tabular}
\end{center}
\caption{Number of elements per wavelength $N_{E}$, number of degrees of freedom $n_{vec}$, size of the global matrix $n_{mat}$, number of time steps $N_{\Delta t}$, and computational cost $n_{comp}$ for a dispersion error of $0.001$ for different finite element methods for the scalar wave equation. The new mass-lumped methods are marked in bold. The numbers are accurate up to two decimal places. }
\label{tab:errDispAcComp}
\end{table}

We will refer to the standard linear mass-lumped finite element method as ML1. The higher-order mass-lumped methods will be referred to as ML[$p$]n[$n$], where $p$ is the degree and $n$ the number of nodes per element. In particular, the degree-2 method of \cite{mulder96} will be referred to as ML2n23 and the two versions of the degree-3 method in \cite{chin99} will be referred to as ML3n50a and ML3n50b. The mass-lumped methods introduced in this paper will be referred to as ML2n15, ML3n32, ML4n60, ML4n61, and ML4n65.

We will also compare the mass-lumped methods with the symmetric interior penalty discontinuous Galerkin (SIPDG) method, introduced and analyzed in [Grote et al. 2006]. For the penalty term, we use the lower bound derived in \cite{geevers17}, since it was shown in \cite{geevers18a} that this penalty term results in a significantly more efficient scheme than the penalty terms based on the more commonly used trace inequality of \cite{warburton03}. The quantities for the computational cost are computed in the same way as for the mass-lumped method, but now $n_0$ denotes the number of basis functions in $\Omega_0$ and $n_{mat}$ is computed as $n_{mat}=\frac{\lambda^3}{|\Omega_0|} \sum_{e\in\mathcal{T}_{\Omega_0}} |U_e|^2|\mathcal{N}(e)|$, where $\mathcal{T}_{\Omega_0}$ are the elements in $\Omega_0$, $|U_e|$ are the number of basis functions per element, and $|\mathcal{N}(e)|=5$ are the number of elements connected with $e$ through a face, including $e$ itself. We will refer to the SIPDG methods of degree 1, 2, 3, and 4, as DG1, DG2, DG3, and DG4, respectively. 

For the time integration, we combine each degree-$p$ finite element method with an order-$2p$ Dablain time integration scheme, since this results in order-$2p$ convergence of the dispersion error. 
 
Figure \ref{fig:errDisptAc} illustrates the relation between the dispersion error and number of elements per wavelength. The dispersion error of the finite element methods converge with order $2p$, which is typical for symmetric finite element methods for eigenvalue approximations, see, for example, \cite{boffi10} and the references therein. Using extrapolation, we obtain formulas for the dispersion error, given in Table \ref{tab:errDisp}. These formulas can be used to determine the required resolution of the mesh given the wavelength and desired accuracy. From the leading constants we can see that the new mass-lumped methods of degree 2 and 3 are more accurate for the same mesh resolution than the SIPDG and existing mass-lumped methods of these orders. The degree-4 mass-lumped method with 65 nodes is slightly more accurate than the versions using 60 and 61 nodes, but is slightly less accurate than the degree-4 discontinuous Galerkin method.

While some methods are more accurate for the same mesh resolution, this does not necessarily mean that these methods are more efficient, since the computational cost per element can greatly differ per method. To get an idea which method is most efficient for a given accuracy, we also look at the relation between the dispersion error and the estimated computational cost. This relation is illustrated in Figure \ref{fig:errDisptAcComp}. The required computational cost of each method for a dispersion error of $0.001$ is also illustrated in Table \ref{tab:errDispAcComp}. 

These results show that the new degree-2 mass-lumped method significantly outperforms the other degree-2 finite element methods, reducing the required computational cost for a given accuracy by one order of magnitude. The new degree-3 method is also significantly more efficient than the other degree-3 methods, reducing the required computational cost by more than a factor 5. These reductions in computational cost can be explained by the improved accuracy for the same mesh resolution, a reduction in the number of degrees of freedom and therefore reduction of the size of the stiffness matrix, and by a smaller number of time steps due to a larger allowed time step size. 

Among the degree-4 finite element methods, the mass-lumped method using 65 nodes performs best, mainly due to a smaller number of required time steps, although these differences are relatively small.

Figure \ref{fig:errDisptAcComp} also indicates that for a dispersion error between $1\%$ and $0.1\%$, the new degree-2 mass lumped method performs best, while for smaller dispersion errors, the new degree-3 mass lumped method is most efficient. When we extrapolate the graphs, we find that the degree-4 mass-lumped method using 65 nodes will only outperform the degree-3 method for a dispersion error below $10^{-5}$.

While the dispersion analysis provides useful information on the efficiency of the numerical methods, it does not include the effect of interpolation errors or inaccurate higher-frequency modes that may contaminate the numerical solution. Furthermore, the estimated computational cost is no perfect measure for the computation time, since the real computation time heavily depends on the implementation of the algorithm and the hardware that is used. In the next section we therefore also show the results of several numerical tests for the mass-lumped methods.

\section{Numerical tests}
\label{sec:tests}

\subsection{Homogeneous Domain}

We first tested and compared the old and new mass-lumped tetrahedral element methods on a homogeneous acoustic model using unstructured tetrahedral meshes. The domain is $[-2,2]\times[-1,1]\times[0,2]\,$km$^3$ and the acoustic wave propagation speed is $c_P:=2\,$km/s. A $3.5$-Hz Ricker wavelet, starting from the peak, was placed at  $\vx_{src}:=(0,0,1000)\,$m, and 56 receivers were placed on a line between $x_r=-1375$ and $x_r=+1375\,$m with a $50$-m interval at $y_r=0\,$m and $z_r=800\,$m. Data were recorded for 0.6$\,$s, counting from the time at which the wavelet peaked, but the computations already started at the negative time -0.6$\,$s when the wavelet is approximately zero.

\begin{figure}[h]
\centering
\begin{subfigure}[b]{0.45\textwidth}
  \includegraphics[width=\textwidth]{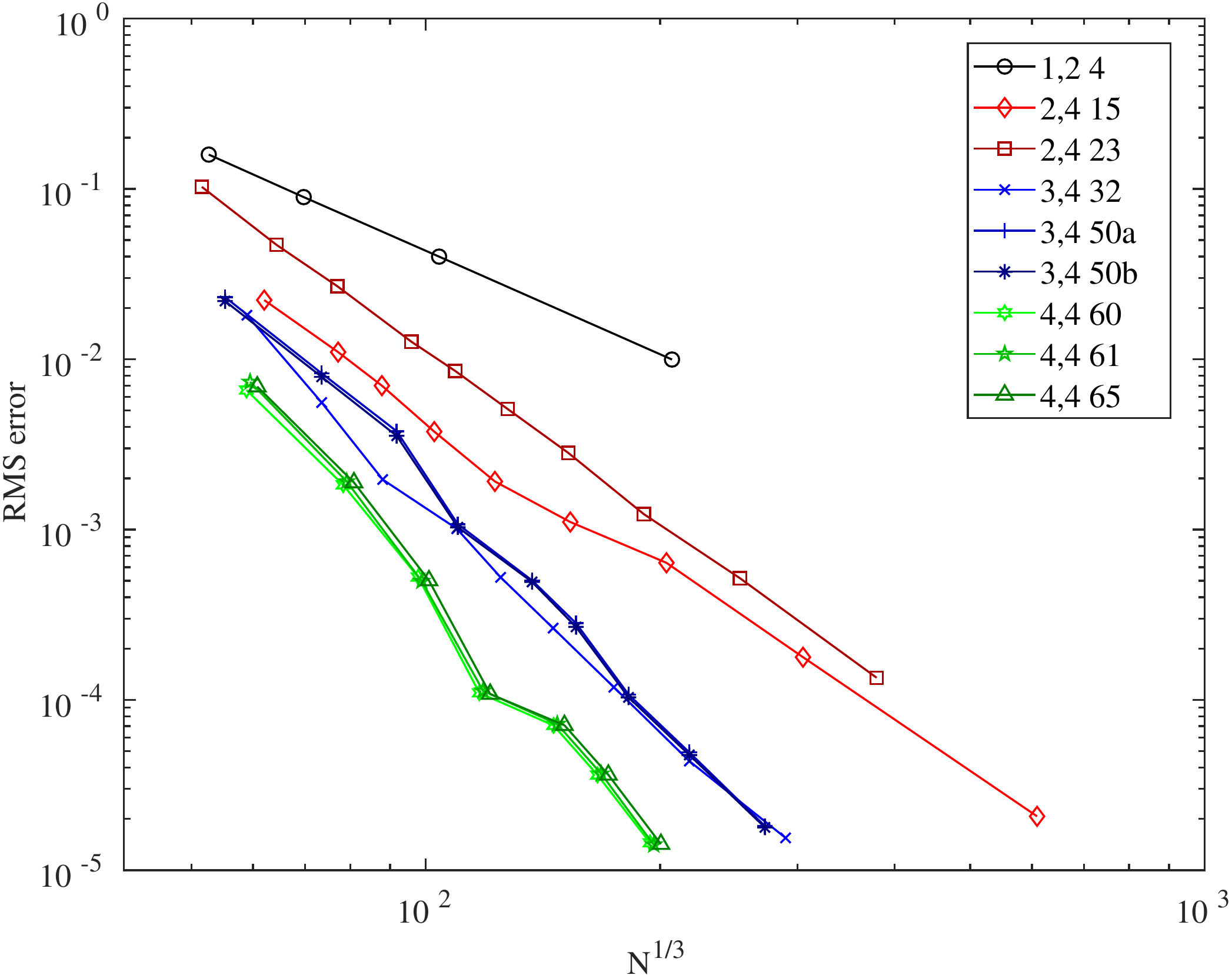}
\end{subfigure} \,\,
\begin{subfigure}[b]{0.45\textwidth}
  \includegraphics[width=\textwidth]{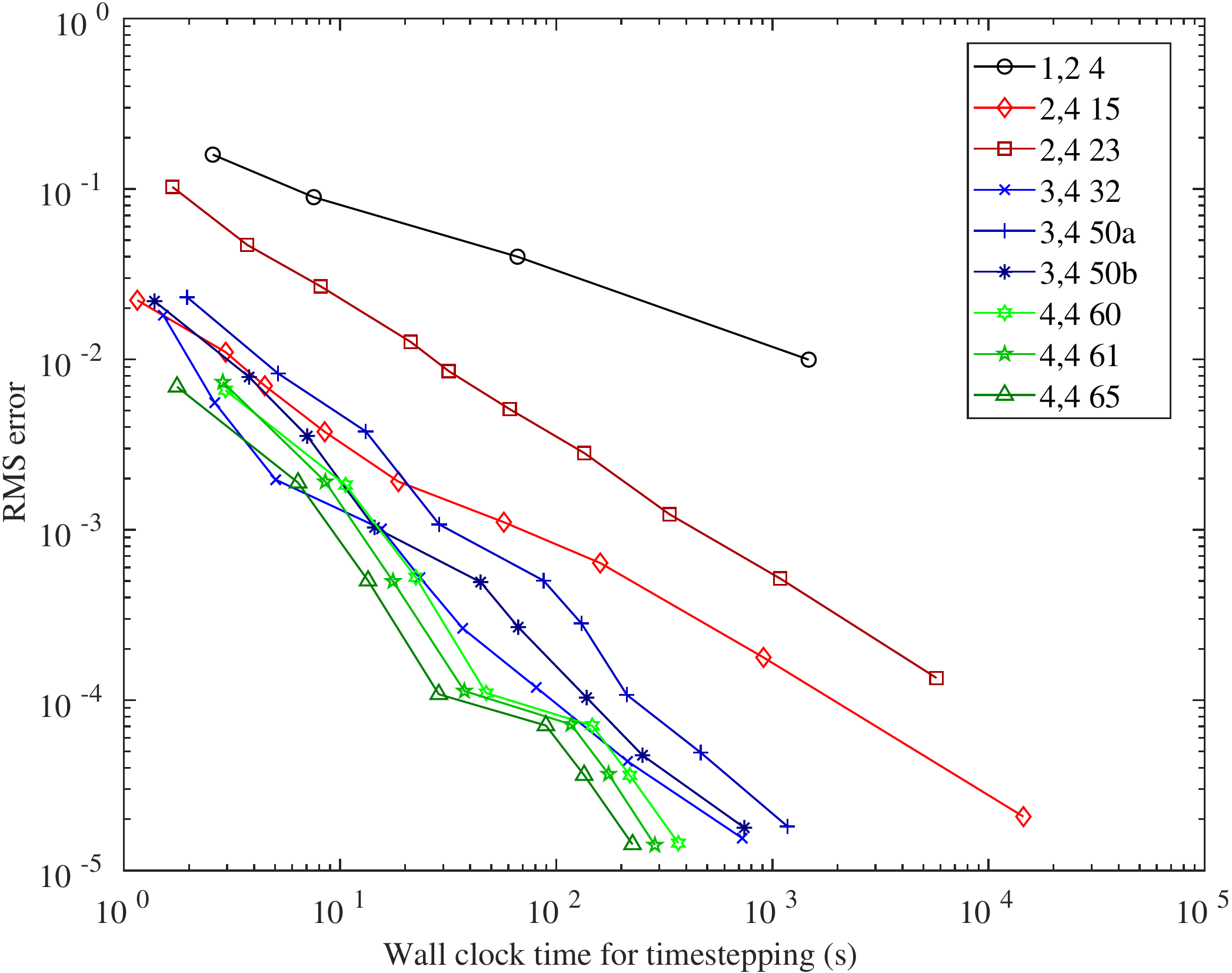}
\end{subfigure}
\caption{RMS errors as a function of the cube root of the number of degrees of freedom (left) and as a function of the wall clock time (right).
In the legend, $p,K\;n$ refers to the element of degree $p$ with $n$ nodes, combined with a $K$-order time-stepping scheme. The older elements, apart from the one with degree 1, have degree 2 with 23 nodes and degree 3 with 50 nodes and two variants called a and b.}
\label{fig:rmsa}
\end{figure}

\begin{table}[h]
\caption{Linear fits of the left graph of Figure \ref{fig:rmsa}.}
\label{tab:errFit}
\begin{center}
\begin{tabular}{l |c ||l |c}
Method			& RMS error 			& Method			& RMS error \\ \hline\hline
1,2 4	 			& $(4.9\times10^2 )N^{(-1/3 \times 2.0)}$ 	& 2,4 15	& $(4.4\times 10^3) N^{(-1/3 \times 3.0)}$ \\ 
  				& 								& 2,4 23 	& $(4.9\times 10^4) N^{(-1/3 \times 3.3)}$ \\ \hline
3,4 32 			& $(9.2\times 10^5) N^{(-1/3 \times 4.4)}$ 	& 4,4 60  	& $(8.6\times 10^6) N^{(-1/3 \times 5.1)}$ \\ 
3,4 50a 			& $(2.9\times 10^6) N^{(-1/3 \times 4.6)}$ 	& 4,4 61  	& $(1.3\times 10^7) N^{(-1/3 \times 5.2)}$ \\
3,4 50b		 	& $(2.6\times 10^6) N^{(-1/3 \times 4.6)}$ 	& 4,4 65  	& $(1.2\times 10^7) N^{(-1/3 \times 5.2)}$ \\
\end{tabular}
\end{center}
\end{table}

The exact solution, in case of an unbounded domain, is given by
\begin{align*}
u(\vx,t) &= \frac{w(t-r/c_P)}{4\pi r}, && t\leq 0.6\,\mathrm{s},
\end{align*}
where $r:=|\vx-\vx_{src}|$ is the distance to the source and $w(t):= (1-2\pi^2f^2t^2)e^{-\pi^2f^2t^2}$ the Ricker-wavelet of peak frequency $f=3.5\,$Hz. This wavelet is zero up to machine precision for $t\leq-0.6\,$s. For the bounded domain, on which we imposed zero Neumann boundary conditions, we add mirror sources to handle the reflections caused by the boundary conditions.

For the implementation of the mass-lumped methods we used the algorithm described in \cite{mulder16}. The time step size is based on the estimates of \cite{mulder14} multiplied by a factor 0.9. The simulations were carried out with OpenMP on 24 cores of two Intel\textsuperscript{\textregistered}~Xeon\textsuperscript{\textregistered} E5-2680~v3 CPUs running at 2.50$\,$GHz. Figure \ref{fig:rmsa} shows the observed root mean square (RMS) errors of the receiver data for the various schemes against the number of degrees of freedom $N$ and against wall clock time. The latter should not be taken too literal because it depends on code implementation, compiler and hardware, and even varies between runs. It can be further reduced by going to single precision, but then it becomes more difficult to measure the errors when they become small. Therefore, we ran a double-precision version of the code when preparing these figures.

Fourth-order time stepping was used for degrees higher than one \cite{dablain86}. For degree 4, we also considered 6th-order time stepping, but the errors were nearly the same as with 4th-order time stepping for the current example.

Power-law fits, given in Table \ref{tab:errFit}, show that the RMS errors converge with approximately order $p+1$ and confirm that the new elements maintain an optimal order of convergence. Figure \ref{fig:rmsa} also shows that the new mass-lumped methods require less degrees of freedom and computation time for the same accuracy. For the degree-2 methods, the difference in wall clock time is up to one order of magnitude, while for the degree-3 methods this difference is up to a factor 2. The degree-4 methods become more efficient for errors below $10^{-3}$.

\subsection{Elastic Salt Model}

We also tested the methods on the more realistic salt model from \cite{kononov12}, made elastic by replacing the water layer at the top by rock. A 3-Hz Ricker wavelet vertical force source was placed on the surface at (2000,2200,0)$\,$m and 25 receivers were placed on a line between $x_r=-1012.5$ and $x_r=7887.5\,$m with a $25$-m interval at $y_r=2200\,$m and $z_r=0\,$m. An illustration of this salt model is given in Figure \ref{fig:saltModel}. Figure \ref{fig:saltSnaps} displays vertical cross sections through the 3D vertical displacement wavefield of the 65-node degree-4 method, clipped at 25\% of its maximum amplitude with red for positive and blue for negative values. Small amplitudes were replaced by the P-velocity to give an impression of the model. Figure \ref{fig:seismogram} also shows seismograms of this method for the displacement in the $x$- and $z$-direction clipped at 2\% of the maximum amplitude. 

\begin{figure}[h]
\centering
\includegraphics[width=0.5\textwidth]{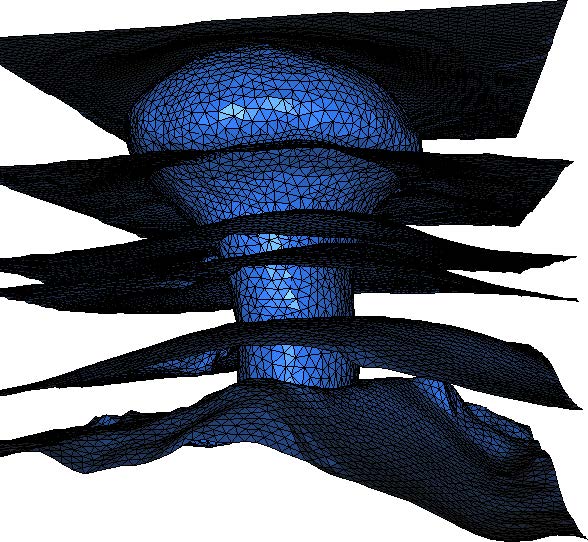}
\caption{Interfaces of the salt model taken from \cite{kononov12}. }
\label{fig:saltModel}
\end{figure}

\begin{figure}[h]
\centering
\begin{subfigure}[b]{0.45\textwidth}
  \includegraphics[width=\textwidth]{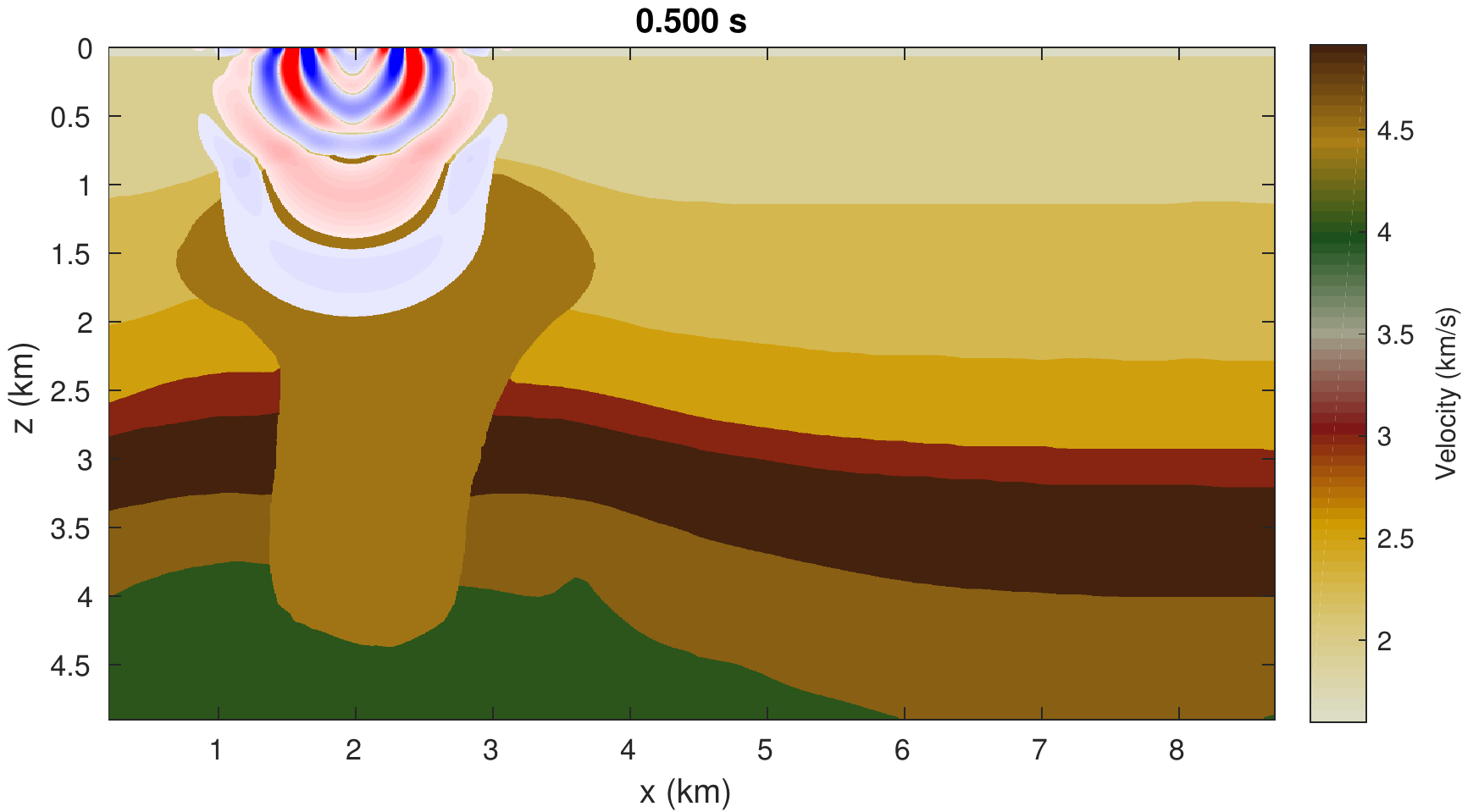}
  \caption{}
\end{subfigure} \,\,
\begin{subfigure}[b]{0.45\textwidth}
  \includegraphics[width=\textwidth]{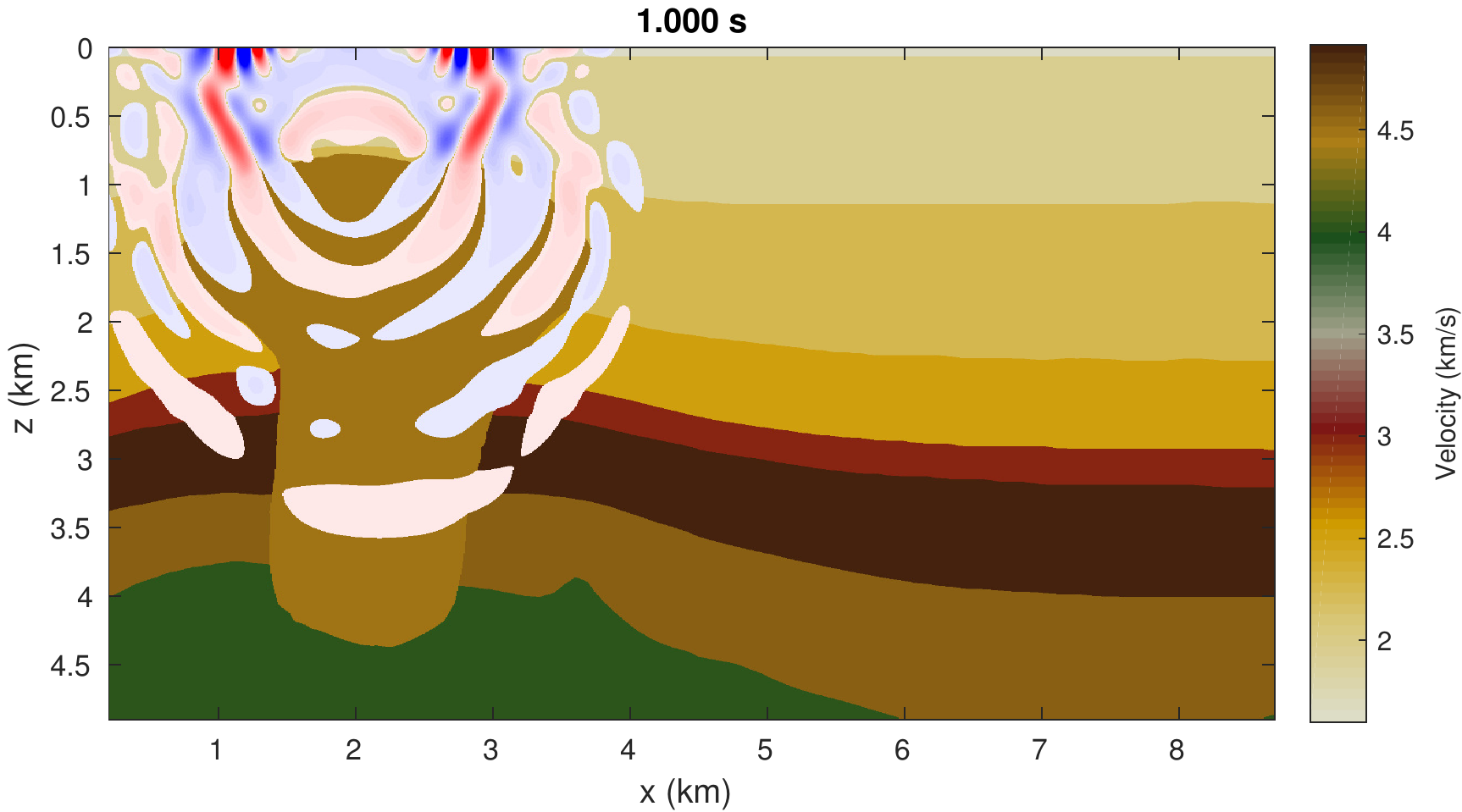}
  \caption{}
\end{subfigure} \\
\begin{subfigure}[b]{0.45\textwidth}
  \includegraphics[width=\textwidth]{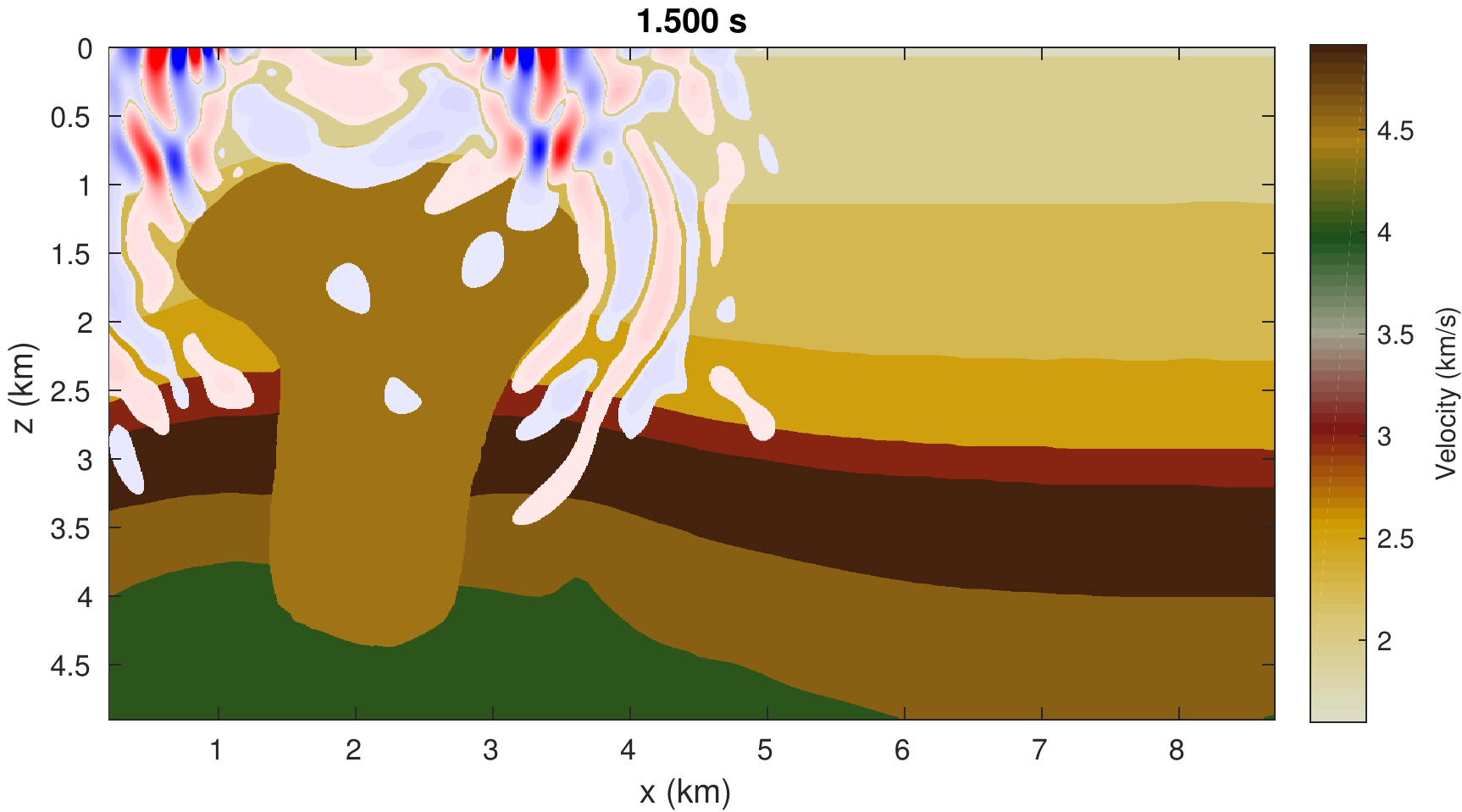}
  \caption{}
\end{subfigure} \,\,
\begin{subfigure}[b]{0.45\textwidth}
  \includegraphics[width=\textwidth]{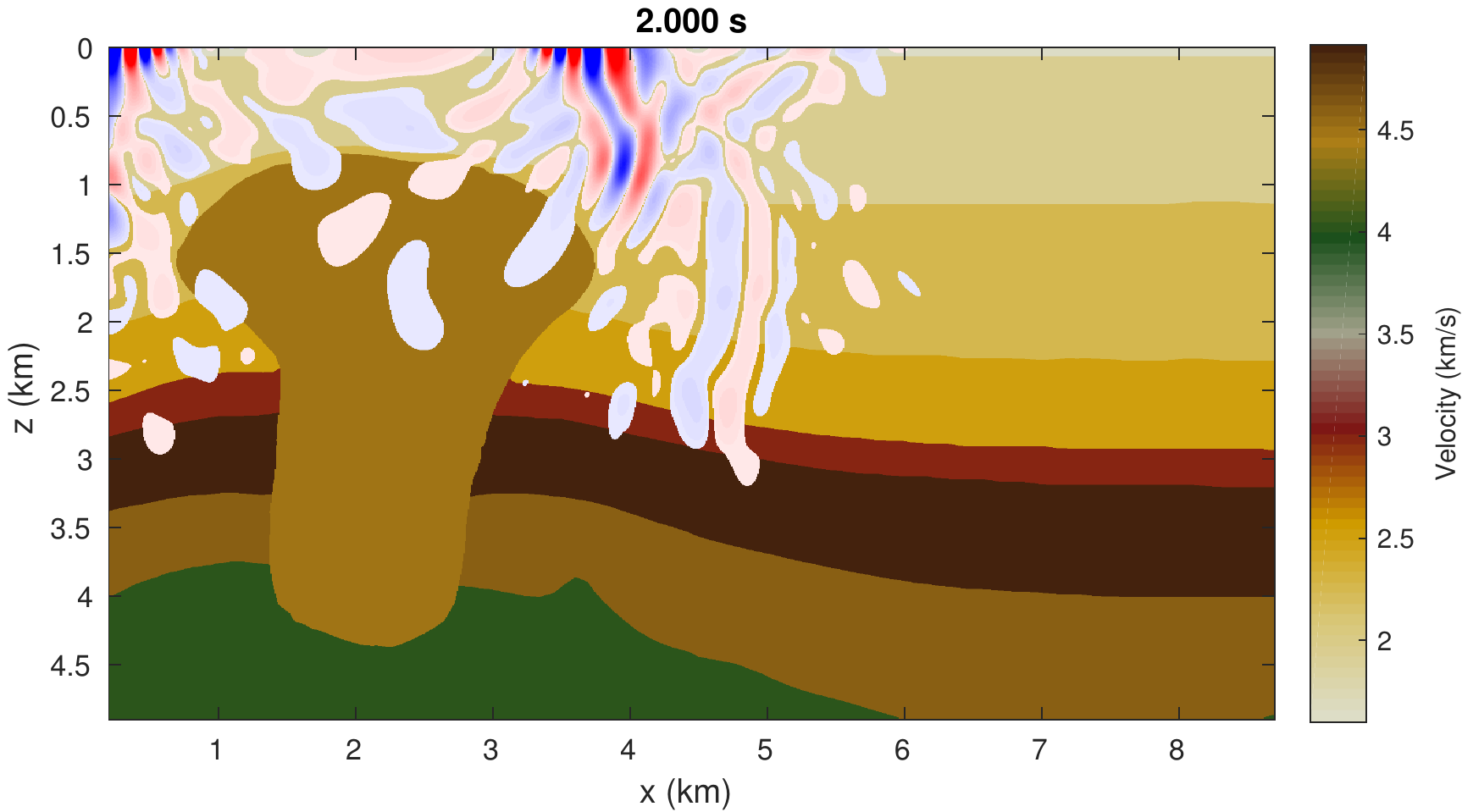}
  \caption{}
\end{subfigure}
\caption{Vertical cross section at $y=\,$2.2$\,$km of the P-velocity with a superimposed snapshot of the vertical displacement of the 65-node degree-4 method after 0.5$\,$s (a), 1.0$\,$s (b), 1.5$\,$s (c) and 2$\,$s (d).}
\label{fig:saltSnaps}
\end{figure}

\begin{figure}[h]
\centering
\begin{subfigure}[h]{0.45\textwidth}
  \includegraphics[width=\textwidth]{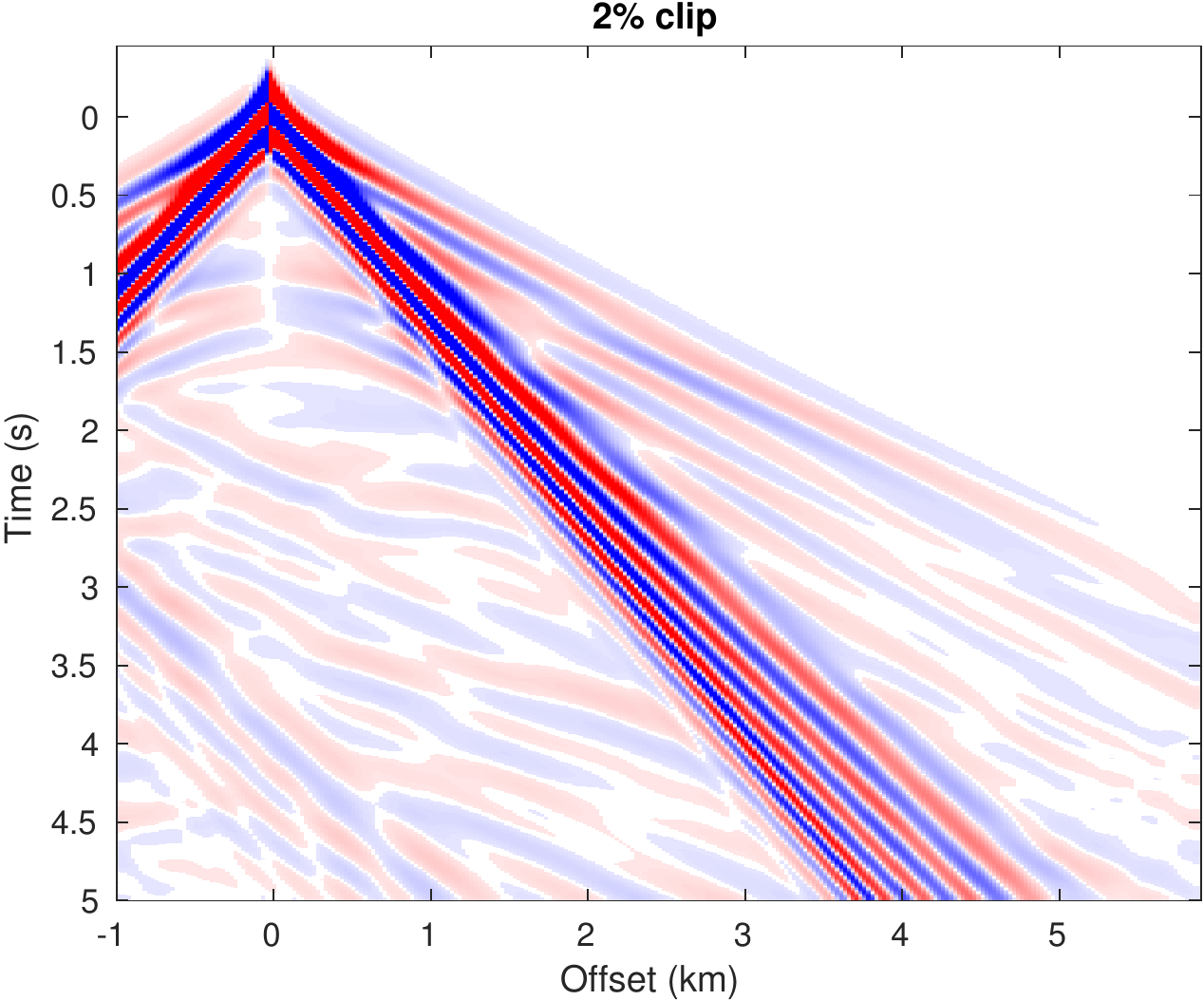}
\end{subfigure} \,\,
\begin{subfigure}[h]{0.45\textwidth}
  \includegraphics[width=\textwidth]{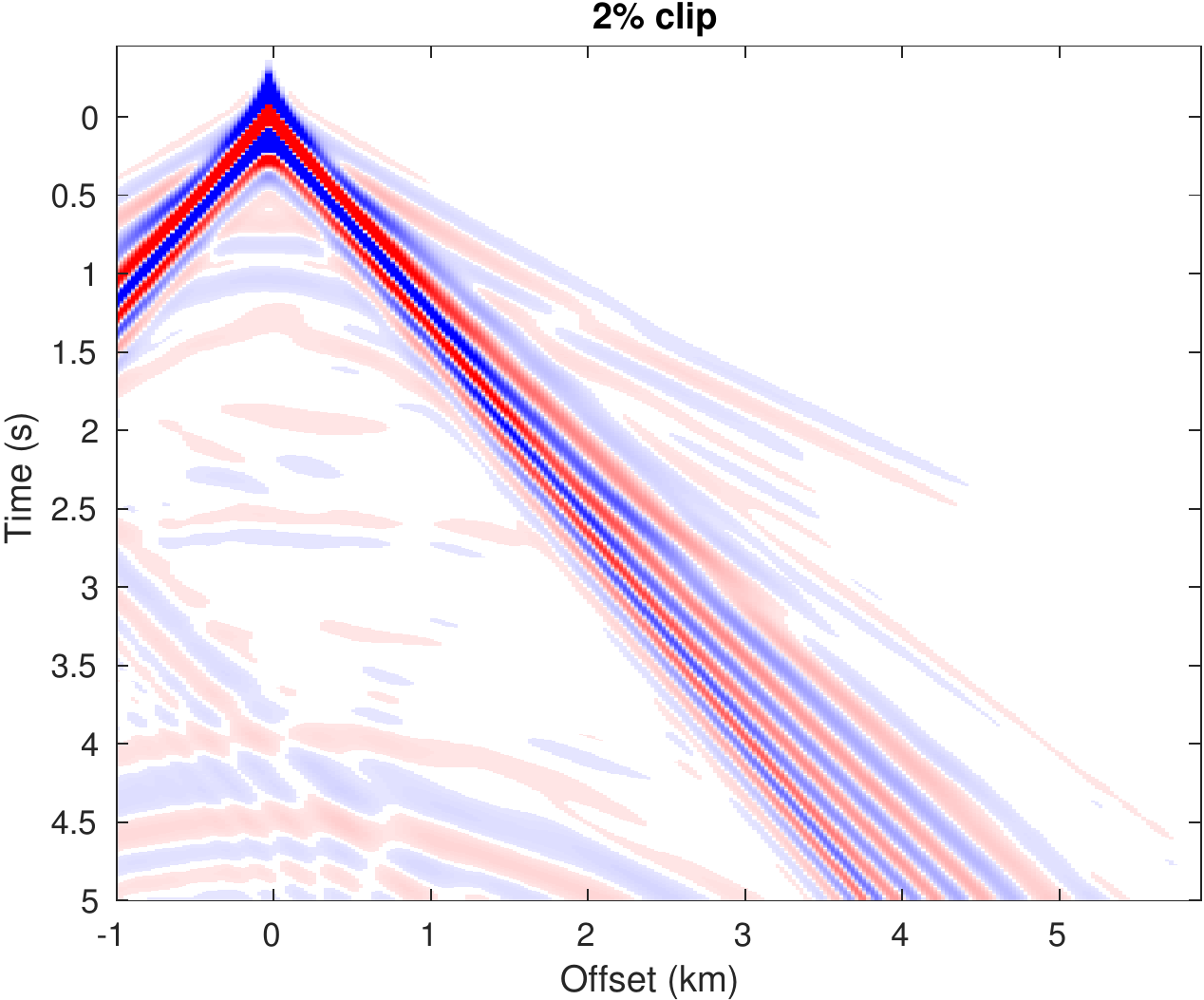}
\end{subfigure}
\caption{Seismogram of the displacement in the $x$- (left) and $z$-direction (right) for the 65-node degree-4 method. }
\label{fig:seismogram}
\end{figure}

\begin{table}[h]
\caption{Estimated relative RMS error of the displacement in the x and z-direction and measured wall clock time. The error is estimated by computing the difference with the ML4n65T4 data. ML[$p$]n[$n$]T[$K$] refers to an element of degree $p$, with $n$ nodes per element, combined with a time stepping scheme of order $K$. There are two versions of the ML3n50 element. ML1T2 is also tested on two refined meshes. New mass-lumped methods are marked in bold.}
\label{tab:errSeis}
\begin{center}
\begin{tabular}{l ||l |l |r}
Method			& RMS-$x$ & RMS-$z$ & time (s)\\ \hline\hline
ML1T2			& $0.50$ & $0.42$ & $ 324$ \\ 
				& $0.36$ & $0.39$ & $ 4230$ \\ 
				& $0.12$ & $0.11$ & $ 47962$ \\ \hline
\textbf{ML2n15T2}	& $0.044$ & $0.048$ & $7358$ \\
ML2n23T2		& $0.058$ & $0.064$ & $28950$ \\ \hline
\textbf{ML3n32T4}	& $0.0014$ & $0.0015$ & $41946$ \\ 
ML3n50aT4		& $0.0016$ & $0.0017$ & $312149$ \\ 
ML3n50bT4		& $0.0016$ & $0.0017$ & $177312$ \\ \hline
\textbf{ML4n60T4}	& $0.000017$ & $0.000019$ & $613945$ \\ 
\textbf{ML4n61T4}	& $0.000013$ & $0.000014$ & $336362$ \\
\textbf{ML4n65T4}	& $0$ & $0$ & $275933$ 
\end{tabular}
\end{center}
\end{table}

Simulations were carried out with the same implementation and in the same environment as for the homogeneous test case. The RMS errors are estimated by taking the traces for the 65-node degree-4 method as the `exact' solution. For the RMS error we use the data of receivers between $x=\,$2.1 and $x=\,$4$\,$km in the time interval $[0,1.8]\,$s. We selected this subset to exclude errors caused by the absorbing boundary layers. To compute the relative RMS error, we divide by the RMS of the data. 

An overview of the RMS errors and the wall-clock time is given in Table \ref{tab:errSeis}. The differences in the traces of the different degree-4 methods is of order $10^{-5}$ and is much smaller then the estimated errors of the lower-degree elements. This indicates that the RMS errors of the degree-4 methods are of order $10^{-5}$ and supports the idea that the degree-4 method can be used to estimate the accuracy of the lower-degree methods in this case. The Table illustrates again that the new degree-2 and degree-3 mass-lumped tetrahedral elements are more efficient than the current ones. The differences in computation time between the degree-4 methods is mainly due to the difference in number of time steps. The ML4n65 allows for a larger time step size than the other variants, which makes it slightly more efficient.


\section{Conclusion}
\label{sec:conclusion}
We developed a less restrictive accuracy condition for the construction of continuous mass-lumped elements, which enabled us to construct several new tetrahedral elements. The new degree-2 and degree-3 tetrahedral elements require 15 and 32 nodes, while the current versions require 23 and 50 nodes per element, respectively. These new elements require less degrees of freedom and allow larger time steps than the current versions. We also developed degree-4 tetrahedral elements with 60, 61, and 65 nodes per element. Mass-lumped tetrahedral elements of this degree had not been found yet. 

A dispersion analysis and numerical examples confirm that the new mass-lumped methods maintain an optimal order of accuracy and show that the new elements are significantly more efficient than the existing ones. In particular, the new degree-2 method is shown to be up to one order of magnitude faster than the current method, while the new degree-3 tetrahedral element results in a speed-up of up to a factor 2 for the same accuracy. The new degree-4 elements outperform the lower-degree elements for an accuracy below $10^{-3}$. Among these degree-4 elements, the one with 65 nodes is  the most efficient, which is mainly due to a larger allowed time step size. The dispersion analysis also shows that the new degree-2 and degree-3 mass-lumped methods require significantly less degrees of freedom and number of time steps than the symmetric interior penalty discontinuous Galerkin methods of the same degree.

We have only considered tetrahedral elements in this paper, but the accuracy condition might also lead to more efficient triangular or higher-dimensional simplicial elements. Furthermore, although we focused only on linear wave propagation problems, mass lumping is useful for solving any type of evolution problem that requires explicit time-stepping.

\bibliographystyle{abbrv}
\bibliography{MLFEM}

\appendix
\section{Nodal Basis Functions}

\begin{lem}
\label{lem:Srep}
Let $e$ be a $d$-simplex in $\mathbb{R}^d$, with $d\geq 0$, and let $s: \mathbb{R}^d\rightarrow \mathbb{R}^d$ be an affine mapping that maps $e$ onto itself. Then $s$ can be represented by a permutation of the barycentric coordinates of $e$. In particular, there exists a permutation function $P:\mathbb{R}^{d+1}\rightarrow\mathbb{R}^{d+1}$ such that 
\begin{align}
\label{eq:Srep}
s(\vx)^* &= P(\vx^*) &&\text{for all }\vx\in\mathbb{R}^d,
\end{align}
where $\vx^*, s(\vx)^*\in\mathbb{R}^{d+1}$ denote the barycentric coordinates of $\vx$ and $s(\vx)$, respectively.
\end{lem}
\begin{proof}
Note that both $\vx\rightarrow s(\vx)^*$ and $\vx\rightarrow P(\vx^*)$ are affine mappings from $\mathbb{R}^d$ to $\mathbb{R}^{d+1}$. It therefore suffices to show that (\ref{eq:Srep}) holds for all vertices of $e$. To do this, let $\vct{v}_i\in\mathbb{R}^d$, for $i=1,\dots,d+1$, denote the vertices of $e$, and let $p$ be the permutation of $(1,\dots,d+1)$ such that $s(\vct{v}_i)=\vct{v}_{p(i)}$. If we define P such that $P(\vct{e}_i)=\vct{e}_{p(i)}$, with $\vct{e}_i\in\mathbb{R}^{d+1}$ the unit vector in direction $i$, then
\begin{align*}
s(\vct{v}_i)^* &= \vct{v}_{p(i)}^* \\
&= \vct{e}_{p(i)} \\
&= P(\vct{e}_{i}) \\
&= P(\vct{v}_i^*)
\end{align*}
for all $i=1,\dots,d+1$. 
\end{proof}

\begin{lem}
\label{lem:nodalBF}
Let $\mathcal{Q}_h$ be defined as in Section \ref{sec:nodalBF}, and let $\vx\in \overline e$, for some $\vx\in\mathcal{Q}_h, e\in\mathcal{T}_h$. If (\ref{eq:confC2a}) is satisfied, then $\phi_e^{-1}(\vx)\in\tilde{\mathcal{Q}}$.
\end{lem}
\begin{proof}
By definition of $\mathcal{Q}_h$, there exists an element $e^*\in\mathcal{T}_h$ and node $\tilde\vx^*\in\tilde{\mathcal Q}$ such that $\vx=\phi_{e^*}(\tilde\vx^*)$, so $\phi_{e^*}^{-1}(\vx)=\tilde\vx^*\in\tilde{\mathcal Q}$. Now construct a reference-to-reference element mapping $s\in\mathcal{S}$ such that $s|_{\phi_{e^*}^{-1}(\overline e^*\cap \overline e)} = \phi_{e}^{-1}\circ\phi_{e^*}|_{\phi_{e^*}^{-1}(\overline e^*\cap \overline e)}$. Then, using (\ref{eq:confC2a}), we can obtain $\phi_{e}^{-1}(\vx) = s\circ\phi_{e^*}^{-1}(\vx) = s(\vx^*) \in \tilde{\mathcal Q}$.
\end{proof}

\begin{thm}
\label{thm:nodalBF}
Let $\mathcal{Q}_h$ be defined as in Section \ref{sec:nodalBF}, and let $\vx\in\mathcal{Q}_h$. If the conditions in (\ref{eq:confC1}) and (\ref{eq:confC2}) are satisfied, then the basis function $w_{\vx}$, given in (\ref{eq:nodalBFD}), is well-defined and continuous and satisfies
\begin{align}
w_{\vx}(\vy)&=\delta_{\vx\vy} &&\text{for all }\vy\in\mathcal{Q}_h,
\label{eq:nodalBFP}
\end{align}
where $\delta_{\vx\vy}$ denotes the Kronecker delta function.
\end{thm}
\begin{proof}
The fact that $w_{\vx}$ is well defined follows immediately from Lemma \ref{lem:nodalBF}.

To prove that $w_{\vx}$ is continuous, let $f=\partial e^-\cap \partial e^+$ be any face adjacent to the elements $e^-,e^+\in\mathcal{Q}_h$. It is sufficient to show that $w_{\vx}|_{\partial e^+\cap f} = w_{\vx}|_{\partial e^-\cap f}$, where $w|_{\partial e\cap f}$ denotes the trace of function $w$ restricted to $e$ on face $f$. Suppose that $\vx\notin \overline e^{-}\cup \overline e^{+}$. Then $w_{\vx}|_{\partial e^-\cap f} = 0 = w_{\vx}|_{\partial e^+\cap f}$.

Now suppose that $\vx\in \overline e^{+}\setminus f$. Then $\vx\notin \overline e^-$ and $\tilde\vx^+\notin\tilde f^+$, where $\tilde{\vx}^+:=\phi_{e^+}^{-1}(\vx)$ and $\tilde f^{+}:=\phi_{e^+}^{-1}(f)$. From (\ref{eq:confC1}) it then follows that $w_{\vx}|_{\partial e^+\cap f} = \tilde w_{\tilde\vx^+}\circ\phi_{e^+}^{-1}|_f = 0 = w_{\vx}|_{\partial e^-\cap f}$.

Finally, suppose that $\vx\in f$. Let $s\in\mathcal{S}$ be a reference-to-reference element mapping such that $s|_{\tilde f^+} = \phi_{e^-}^{-1}\circ\phi_{e^+}|_{\tilde f^+}$. We can then derive
\begin{align*}
w_{\vx}|_{\partial e^+\cap f} &= \tilde w_{\tilde\vx^+}\circ\phi_{e^+}^{-1}|_f \\
&= \tilde w_{s(\tilde\vx^+)}\circ s \circ\phi_{e^+}^{-1}|_f \\
&= \tilde w_{\phi_{e^-}^{-1}(\vx)}\circ\phi_{e^-}^{-1}|_f \\
&= w_{\vx}|_{\partial e^-\cap f},
\end{align*}
where the second line follows from (\ref{eq:confC2b}).

Since $w_{\vx}$ is continuous, $w_{\vx}(\vy)$ is well-defined for any $\vy\in\Omega$. To prove property (\ref{eq:nodalBFP}), suppose there is an element $e\in\mathcal{T}_h$ such that $\vx,\vy\in \overline e$. Define $\tilde\vx:=\phi_e^{-1}(\vx)$ and $\tilde\vy:=\phi_e^{-1}(\vy)$. Then $w_{\vx}(\vy)=\tilde w_{\tilde\vx}(\tilde\vy) = \delta_{\tilde\vx\tilde\vy}=\delta_{\vx\vy}$.

Now suppose that there is no element $e$ such that $\vx,\vy\in \overline e$. Then $\vx\neq\vy$ and there exists an element $e$ such that $\vy\in \overline e$ and $\vx\notin \overline e$. By definition of $w_{\vx}$ it then follows that $w_{\vx}(\vy)=0=\delta_{\vx\vy}$.
\end{proof}

\section{Variants of the Degree-Four Mass-Lumped Tetrahedral Element}
\label{sec:varML4}

\begin{table}[h]
\caption{Degree-$4$ mass-lumped tetrahedral element with $60$ nodes.}
\label{tab:ML4n60}
\begin{center}
{\tabulinesep=0.5mm
\begin{tabu}{c r l l}
Nodes					& $n$ 	& $\omega$ 	& parameters  \\ \hline
$\{(0,0,0)\}	$				& $4$	& $0.00009319146955767176$		& - \\
$\{(a,0,0)\}$				& $12$	& $0.0004829332376473431$		& $0.1614865833496676$ \\
$\{(\frac12,0,0)\}$			& $6$	& $0.0002005503792135920$		& - \\
$\{(b_1,b_1,0)\}$			& $12$	& $0.002003104085841525$		& $0.1490219288469598$ \\
$\{(b_2,b_2,0)\}$			& $12$	& $0.001126849366800016$		& $0.3944591972171783$ \\
$\{(c_1,c_1,c_1)\}$			& $4$	& $0.009159244489996298$		& $0.1302058846372564$ \\ 
$\{(d,d,\frac12-d)\}$			& $6$	& $0.006725322654059780$		& $0.06386116838612691$ \\ 
$\{(c_2,c_2,c_2)\}$			& $4$	& $0.01118676108633598$		& $0.3012179234079087$ \\ \hline
 \multicolumn{4}{c}{$U=\mathcal{P}_4 \oplus \mathcal{B}_f\mathcal{P}_2 \oplus \mathcal{B}_e(\mathcal{P}_2\oplus\mathcal{B}_f)$}  \\ 
 \multicolumn{4}{c}{$=\{x_1,x_1^2x_2,x_1^2x_2^2,\beta_fx_1, \beta_fx_1x_2, \beta_ex_1,\beta_ex_1x_2,\beta_e\beta_f \} $}  \\ \hline
 \multicolumn{4}{c}{$U\otimes\mathcal{P}_2=\{x_1, x_1^2x_2, x_1^3x_2^2, x_1^3x_2^3, \beta_fx_1, \beta_fx_1^2x_2, \beta_fx_1^2x_2^2,\beta_f^2x_1, \dots$}  \\  
  \multicolumn{4}{c}{$\dots, \beta_ex_1,\beta_ex_1^2x_2,\beta_ex_1^2x_2^2, \beta_e\beta_fx_1,\beta_e\beta_fx_1x_2, \beta_e^2x_1 \}$}  \\  \hline
\end{tabu}}
\end{center}
\end{table}

\begin{table}[h]
\caption{Degree-$4$ mass-lumped tetrahedral element with $61$ nodes.}
\label{tab:ML4n61}
\begin{center}
{\tabulinesep=0.5mm
\begin{tabu}{c r l l}
Nodes					& $n$ 	& $\omega$ 	& parameters  \\ \hline
$\{(0,0,0)\}	$				& $4$	& $0.0001593069370906064$		& - \\
$\{(a,0,0)\}$				& $12$	& $0.0004461325181676239$		& $0.2001628104707848$ \\
$\{(\frac12,0,0)\}$			& $6$	& $0.0003715829945705960$		& - \\
$\{(b_1,b_1,0)\}$			& $12$	& $0.001884294964657102$		& $0.1397350972238366$ \\
$\{(b_2,b_2,0)\}$			& $12$	& $0.001545425606069384$		& $0.4319436235177682$ \\
$\{(c_1,c_1,c_1)\}$			& $4$	& $0.008841425190569096$		& $0.1282209316290979$ \\ 
$\{(d,d,\frac12-d)\}$			& $6$	& $0.006891012924401557$		& $0.08742182088664353$ \\ 
$\{(c_2,c_2,c_2)\}$			& $4$	& $0.007499563520517103$		& $0.3124061452070811$ \\ 
$\{(\frac14,\frac14,\frac14)\}$	& $1$	& $0.01057967149339721$		& - \\ \hline
 \multicolumn{4}{c}{$U=\mathcal{P}_4 \oplus \mathcal{B}_f\mathcal{P}_2 \oplus \mathcal{B}_e(\mathcal{P}_2\oplus\mathcal{B}_f\oplus\mathcal{B}_e)$}  \\ 
 \multicolumn{4}{c}{$=\{x_1,x_1^2x_2,x_1^2x_2^2,\beta_fx_1, \beta_fx_1x_2, \beta_ex_1,\beta_ex_1x_2,\beta_e\beta_f,\beta_e^2\} $}  \\ \hline
 \multicolumn{4}{c}{$U\otimes\mathcal{P}_2=\{x_1, x_1^2x_2, x_1^3x_2^2, x_1^3x_2^3, \beta_fx_1, \beta_fx_1^2x_2, \beta_fx_1^2x_2^2,\beta_f^2x_1, \dots$}  \\  
  \multicolumn{4}{c}{$\dots, \beta_ex_1,\beta_ex_1^2x_2,\beta_ex_1^2x_2^2, \beta_e\beta_fx_1,\beta_e\beta_fx_1x_2, \beta_e^2x_1,\beta_e^2x_1x_2\}$}  \\  \hline
\end{tabu}}
\end{center}
\end{table}

\end{document}